\documentclass[10pt,a4paper,twoside,english]{article}
\usepackage{mathpazo}
\usepackage[T1]{fontenc}
\usepackage[latin9]{inputenc}
\usepackage{color}
\usepackage{amsthm}
\usepackage{amsmath}
\usepackage{graphicx}
\usepackage{amssymb}

\theoremstyle{plain}
 \theoremstyle{definition}
 \newtheorem*{defn*}{Definition}
\newtheorem{thm}{Theorem}[section]
  \theoremstyle{remark}
  \newtheorem{rem}[thm]{Remark}
  \theoremstyle{plain}
  \newtheorem{prop}[thm]{Proposition}
  \theoremstyle{remark}
  \newtheorem{notation}[thm]{Notation}
  \theoremstyle{plain}
  \newtheorem{cor}[thm]{Corollary}
  \theoremstyle{definition}
  \newtheorem{defn}[thm]{Definition}
  \theoremstyle{plain}
  \newtheorem{lem}[thm]{Lemma}

\usepackage{babel}
\addto\extrasenglish{\providecommand{\fg}{\ifdim\lastskip>\z@\unskip\fi~\frqq}}

\begin{document}

\title{Semi-classical behaviour of Schr\"{o}dinger\'{}s dynamics : revivals of
wave packets on hyperbolic trajectory}

\author{Olivier Labl\'{e}e}

\date{19 July 2010}
\maketitle
\selectlanguage{english}%
\begin{abstract}
The aim of this paper is to study the semi-classical behaviour of
Schr\"{o}dinger\'{}s dynamics for an one-dimensional quantum Hamiltonian
with a classical hyperbolic trajectory. As in the regular case (elliptic
trajectory), we prove, that for an initial wave packets localized
in energy, the dynamics follows the classical motion during short
time. This classical motion is periodic and the period $T_{hyp}$
is order of \foreignlanguage{english}{{\normalsize $\left|\ln h\right|$}}.
And, for large time, a new period $T_{rev}$ for the quantum dynamics
appears : the initial wave packets form again at $t=T_{rev}$. Moreover
for the time $t=\frac{p}{q}T_{rev}$ a fractionnal revivals phenomenon
of the initial wave packets appears : there is a formation of a finite
number of clones of the original wave packet.
\end{abstract}
Schr\"{o}dinger\'{}s dynamics, revivals of wave packets, semi-classical analysis,
hyperbolic trajectory, Schr\"{o}dinger operator with double wells potential.

\tableofcontents{}

\section{Introduction}

\subsection{Context and motivation}

For $P_{h}$ a pseudo-differential operator (here $h>0$ is the semi-classical
parameter), and for $\psi_{0}$ an inital state the quantum dynamics
is governed by the famous Schr\"{o}dinger equation : \begin{eqnarray*}
ih\frac{\partial\psi(t)}{\partial t}=P_{h}\psi(t).\end{eqnarray*}
In this paper, we present a detailed study, in the semi-classical
regime $h\rightarrow0$, of the behaviour of Schr\"{o}dinger's dynamics
for an one-dimensional quantum Hamiltonian \[
P_{h}\,:\, D\left(P_{h}\right)\subset L^{2}\left(\mathbb{R}\right)\rightarrow L^{2}\left(\mathbb{R}\right)\]
with a classical hyperbolic trajectory : the principal symbol $p\in\mathcal{C}^{\infty}(\mathbb{R}^{2},\mathbb{R})$
of $P_{h}$ has a hyperbolic non-degenerate singularity. 

Dynamics in the regular case and for elliptic non-degenerate singularity
have been the subject of many research in physics \textbf{{[}Av-Pe{]},
{[}LAS{]}, {[}Robi1{]}}, \textbf{{[}Robi2{]}, {[}BKP{]}}, \textbf{{[}Bl-Ko{]}
}and, more recently in mathematics \textbf{{[}Co-Ro{]}}, \textbf{{[}Rob{]},
{[}Pau1{]}}, \textbf{{[}Pau2{]}, {[}Lab2{]}}. The strategy to understand
the long times behaviour of dynamics is to use the spectrum of the
operator $P_{h}$. In the regular case, the spectrum of $P_{h}$ is
given by the famous Bohr-Sommerfeld rules (see for example \textbf{{[}He-Ro{]},}
\textbf{{[}Ch-VuN{]}}, \textbf{{[}Col8{]}}) : in first approximation,
the spectrum of $P_{h}$ in a compact set is a sequence of real numbers
with a gap of size $h.$ The classical trajectories are periodic and
supported on elliptic curves.

In the case of hyperbolic singularity we have a non-periodic trajectory
supported on a ''height'' figure (see figure 2). The spectrum near
this singularity is more complicated than in the regular case. Y.
Colin de Verdi\`{e}re and B. Parisse give an implicit singular Bohr-Sommerfeld
rules for hyperbolic singularity (\textbf{{[}Co-Pa1{]}},\textbf{ {[}Co-Pa2{]}
}and\textbf{ {[}Co-Pa3{]}}).\textbf{ }This quantization formula is
too implicit for using it directly in our motivation. In \textbf{{[}Lab3{]}
}we have an explicit description of the spectrum for an one-dimesional
pesudo-differential operator near a hyperbolic non-degenerate singularity.

\subsection{Results}

With above description, we propose a study of quantum dynamics for
large times $\left(\gg\left|\ln h\right|\right).$ We prove that for
a localized initial state, at the begining the dynamics is periodic
with a period equal to \foreignlanguage{english}{$T_{hyp}=C\left|\ln h\right|$}
(see corollary 5.12). This period $T_{hyp}$ corresponds to the classical
Hamiltonian flow period. 

Next for large time scale, a new period $T_{rev}$ of the quantum
dynamics appears : this is the revivals phenomenon (like in regular
case \textbf{{[}Co-Ro{]}}, \textbf{{[}Rob{]}, {[}Pau1{]}}, \textbf{{[}Pau2{]},
{[}Lab2{]}}). For $t=T_{hyp}$ the packet relocalize in the form of
a quantum revival. 

We have also the phenomenon of fractional revivals of initial wave
packets for time $t=\frac{p}{q}T_{rev}$, with $\frac{p}{q}\in\mathbb{Q}$
: there is a formation of a finite number of clones of the original
wave packet $\psi_{0}$ with a constant amplitude (see theorems 6.18,
6.19 \& 6.20) and differing in the phase plane from the initial wave
packet by fractions $\frac{p}{q}T_{hyp}$ (see theorem 6.15).

\subsection{Paper organization}

The paper is organized as follows. In section 2 we give some preliminaries
about the strategy for analyse the dynamics of a quantum Hamiltonian.
In this section we define a simple way to understand the evolution
of $t\mapsto\psi(t)$ by the autocorrelation function :\[
\mathbf{c}(t):=\left|\left\langle \psi(t),\psi_{0}\right\rangle _{\mathcal{H}}\right|.\]
In section 3 we describe the hyperbolic singularities mathematical
context; we also recall the principal theorem of \textbf{{[}Lab3{]}}.
This theorem provides the spectrum of the operator $P_{h}$ near the
singularity. Section 4 is devoted to define an initial wave packets
$\psi_{0}$ localized in energy. In part 5 we prove that the quantum
dynamics follows the classical motion during short time (see corollary
5.12). This classical motion is periodic and the period $T_{hyp}$
is order of \foreignlanguage{english}{$\left|\ln h\right|$}. In the
last part (part 6) we detail the analysis of revivals phenomenon,
see theorem 6.7 for full-revival theorem and see theorem 6.15 for
fractionnal-revivals phenomenon.

\section{Quantum dynamics and autocorrelation function}

\subsection{The quantum dynamics}

For a quantum Hamiltonian $P_{h}\,:\, D\left(P_{h}\right)\subset\mathcal{H}\rightarrow\mathcal{H}$,
$\mathcal{H}$ is a Hilbert space, the Schr\"{o}dinger dynamics is governed
by the Schr\"{o}dinger equation : \begin{eqnarray*}
ih\frac{\partial\psi(t)}{\partial t}=P_{h}\psi(t).\end{eqnarray*}
With the functional calculus, we can reformulate this equation with
the unitary group $U(t)=\left\{ e^{-i\frac{t}{h}P_{h}}\right\} _{t\in\mathbb{R}}.$
Indeed, for a initial state $\psi_{0}\in\mathcal{H}$, the evolution
given by : \begin{eqnarray*}
\psi(t)=U(t)\psi_{0}\in\mathcal{H}.\end{eqnarray*}

\subsection{Return and autocorrelation function}

We now introduce a simple tool to understand the behaviour of the
vector \foreignlanguage{english}{$\psi(t)$} : a quantum analog or
the Poincar\'{e} return function. 
\begin{defn*}
The quantum return functions of the operator $P_{h}$ and for an initial
state $\psi_{0}$ is defined by : \[
\mathbf{r}(t):=\left\langle \psi(t),\psi_{0}\right\rangle _{\mathcal{H}};\]
and the autocorrelation function is defined by :\[
\mathbf{c}(t):=\left|\mathbf{r}(t)\right|=\left|\left\langle \psi(t),\psi_{0}\right\rangle _{\mathcal{H}}\right|.\]

\end{defn*}
The previous function measures the return on initial state. This function
is the overlap of the time dependent quantum state $\psi(t)$ with
the initial state $\psi_{0}.$ Since the initial state $\psi_{0}$
is normalized, the autocorrelation function takes values in the compact
set $[0,1].$ Then, if we have an orthonormal basis of eigenvectors
$\left(e_{n}\right)_{n\in\mathbb{N}}$ : \[
P_{h}e_{n}=\lambda_{n}(h)e_{n}\]
with\[
\lambda_{1}(h)\leq\lambda_{2}(h)\leq\cdots\leq\lambda_{n}(h){\displaystyle \rightarrow}+\infty;\]
we get, for all integer $n$ \[
\left(e^{-i\frac{t}{h}P_{h}}\right)e_{n}=\left(e^{-i\frac{t}{h}\lambda_{n}(h)}\right)e_{n}.\]
So for a initial vector $\psi_{0}\in D(P_{h})\subset\mathcal{H},$
let us denote by $(c_{n})_{n\in\mathbb{N}}=(c_{n}(h))_{n\in\mathbb{N}}$
the sequence of $\ell^{2}(\mathbb{N})$ given $(c_{n})_{n}=\pi\left(\psi_{0}\right)$,
where $\pi$ is the projector (unitary operator) : \[
\pi:\left\{ \begin{array}{cc}
\mathcal{H}\rightarrow\ell^{2}(\mathbb{N})\\
\\\psi\mapsto<\psi,e_{n}>_{\mathcal{H}}.\end{array}\right.\]
Then, for all $t\geq0$ we have \begin{eqnarray*}
\psi(t)=U(t)\psi_{0}=\left(e^{-i\frac{t}{h}P_{h}}\right)\left({\displaystyle \sum_{n\in\mathbb{N}}c_{n}e_{n}}\right)\end{eqnarray*}
\[
={\displaystyle \sum_{n\in\mathbb{N}}c_{n}e^{-i\frac{t}{h}\lambda_{n}(h)}e_{n}}.\]
So, for all $t\geq0$ we obtain \[
\mathbf{r}(t)={\displaystyle \sum_{n\in\mathbb{N}}\left|c_{n}\right|^{2}e^{-i\frac{t}{h}\lambda_{n}(h)}};\;\mathbf{c}(t)=\left|\sum_{n\in\mathbb{N}}\left|c_{n}\right|^{2}e^{-i\frac{t}{h}\lambda_{n}(h)}\right|.\]

\subsection{Strategy for study the autocorrelation function }

The strategy, performed by the physicists (\textbf{{[}Av-Pe{]}, {[}LAS{]},
{[}Robi1{]}}, \textbf{{[}Robi2{]}, {[}BKP{]}}, \textbf{{[}Bl-Ko{]}})
is the following :
\begin{enumerate}
\item We define a initial vector $\psi_{0}={\displaystyle \sum_{n\in\mathbb{N}}c_{n}e_{n}}$
localized in the following sense : the sequence \foreignlanguage{english}{$(c_{n})_{n\in\mathbb{N}}$}
is localized close to a quantum number $n_{0}$ (depends on $h$ and
a energy level $E\in\mathbb{R}$).
\item Next, the idea is to expand by a Taylor formula's the eigenvalues
\foreignlanguage{english}{$\lambda_{n}(h)$} around the energy level
$E$ :\[
\lambda_{n}(h)=\lambda_{n_{0}}(h)+\lambda_{n_{0}}^{\prime}(h)\left(n-n_{0}\right)+\frac{\lambda_{n_{0}}^{\prime\prime}(h)}{2}\left(n-n_{0}\right)^{2}+\frac{\lambda_{n_{0}}^{(3)}(h)}{6}\left(n-n_{0}\right)^{3}+\cdots\]
(here $\lambda_{n_{0}}(h)$ is the closest eigenvalue to $E$), hence
we get for all $t\geq0$ \[
\mathbf{c}(t)=\left|\sum_{n\in\mathbb{N}}\left|c_{n}\right|^{2}e^{-it\left[\frac{\lambda_{n_{0}}^{\prime}(h)}{h}\left(n-n_{0}\right)+\frac{\lambda_{n_{0}}^{\prime\prime}(h)}{2h}\left(n-n_{0}\right)^{2}+\frac{\lambda_{n_{0}}^{(3)}(h)}{6h}\left(n-n_{0}\right)^{3}+\cdots\right]}\right|.\]

\item And, for small values of $t$, the first approximation of the autocorrelation
function $\mathbf{c}(t)$ is the function \[
\mathbf{c}_{1}(t):=\left|\sum_{n\in\mathbb{N}}\left|c_{n}\right|^{2}e^{-it\frac{\lambda_{n_{0}}^{\prime}(h)}{h}\left(n-n_{0}\right)}\right|;\]
and for larger values of $t$, the order 2-approximation is\[
\mathbf{c}_{2}(t):=\left|\sum_{n\in\mathbb{N}}\left|c_{n}\right|^{2}e^{-it\left[\frac{\lambda_{n_{0}}^{\prime}(h)}{h}\left(n-n_{0}\right)+\frac{\lambda_{n_{0}}^{\prime\prime}(h)}{2h}\left(n-n_{0}\right)^{2}\right]}\right|.\]

\end{enumerate}
In section 5 we study the function $t\mapsto\mathbf{c}_{1}(t)$ and
in section 6 we study $t\mapsto\mathbf{c}_{2}(t).$

\section{The context of hyperbolic singularity}

\subsection{Link between spectrum and geometry : semi-classical analysis }

For explain the philosophy of semi-classical analysis start by an
example : for a real number $E>0$; the equation\[
-\frac{h^{2}}{2}\Delta_{g}\varphi=E\varphi\]
(where $\Delta_{g}$ denotes the Laplace-Beltrami operator on a Riemaniann
manifold $(M,g)$) admits the eigenvectors $\varphi_{k}$ as solution
if\[
-\frac{h^{2}}{2}\lambda_{k}=E.\]
Hence if $h\rightarrow0^{+}$ then $\lambda_{k}\rightarrow+\infty$.
So there exists a correspondence between the semi-classical limit
($h\rightarrow0^{+}$) and large eigenvalues. 

The asympotic's of large eigenvalues for the Laplace-Beltrami operator
$\Delta_{g}$ on a Riemaniann manifold $(M,g)$, or more generally
for a pseudo-differential operator $P_{h}$, is linked to a symplectic
geometry : the phase space geometry. This is the same phenomenon between
quantum mechanics (spectrum, operator algebra) and classical mechanics
(length of periodic geodesics, symplectic geometry). More precisely,
for a pseudo-differential operator $P_{h}$ on $L^{2}(M)$ with a
principal symbol $p\in\mathcal{C}^{\infty}\left(T^{\star}M\right)$,
there exist a link between the geometry of the foliation $\left(p^{-1}(\lambda)\right)_{\lambda\in\mathbb{R}}$
and the spectrum of the operator $P_{h}$. Indeed, we have the famous
result :\[
\left(P_{h}-\lambda I_{d}\right)u_{h}=O(h^{\infty})\]
then\[
MS(u_{h})\subset p^{-1}(\lambda);\]
where $MS(u_{h})\subset T^{\star}M$ denote the microsupport of the
function $u_{h}\in L^{2}(M)$.

\subsection{Hyperbolic singularity}

In dimension one, a point $(x_{0},\xi_{0})\in T^{\star}M$ is a non-degenerate
hyperbolic singularity of the symbol function $p\in\mathcal{C}^{\infty}(T^{\star}M)$
if and only if :
\begin{enumerate}
\item $dp(x_{0},\xi_{0})=0$;
\item the eigenvalues of the Hessian matrix $\nabla^{2}p(x_{0},\xi_{0})$
are pairwise distinct;
\item if, in some local symplectic coordinates $(x,\xi)$ the algebra spanned
by $\nabla^{2}p(x_{0},\xi_{0})$ has a basis of the form $q=x\xi.$ \end{enumerate}
\begin{rem}
There exists analogue definition for completely integrable systems,
see for example the book of San V\~u Ng\d oc \textbf{{[}VuN{]}}.
\end{rem}
The canonical example in dimension 1 is the Schr\"{o}dinger operator with
double wells potential : \[
{\displaystyle P_{h}}=-\frac{h^{2}}{2}\Delta+V;\]
we assume $V\in\mathcal{C}^{\infty}(\mathbb{R}),\,$ for all $x\in\mathbb{R},\, V(x)\geq C$
and ${\displaystyle \lim_{|x|\rightarrow\infty}V(x)=+\infty}$. Here
the principal symbol of $P_{h}$ is the function \[
p(x,\xi)=\frac{\xi^{2}}{2}+V(x).\]
With the previous hypotheses on the potential $V$, the operator $P_{h}$
is self-adjoint, the spectrum is a sequence of real numbers $\left(\lambda_{n}(h)\right)_{n\geq0}$
and the eigenvectors $\left(e_{n}\right)_{n\geq0}$ be an orthonormal
basis of the Hilbert space $L^{2}(\mathbb{R})$ (for example see the
survey\textbf{ {[}Lab1{]}}). We also suppose that the potential $V$
admits exactly one local non-degenerate maximum. Without loss generality,
we may suppose \[
V(0)=0;\, V'(0)=0;\, V''(0)<0.\]
Then the foliation associated to the principal symbol $p(x,\xi)=\xi^{2}/2+V(x)$
admits a singular fiber $\Lambda_{0}=p^{-1}(0).$

\selectlanguage{english}%
\begin{center}
\includegraphics[scale=0.33]{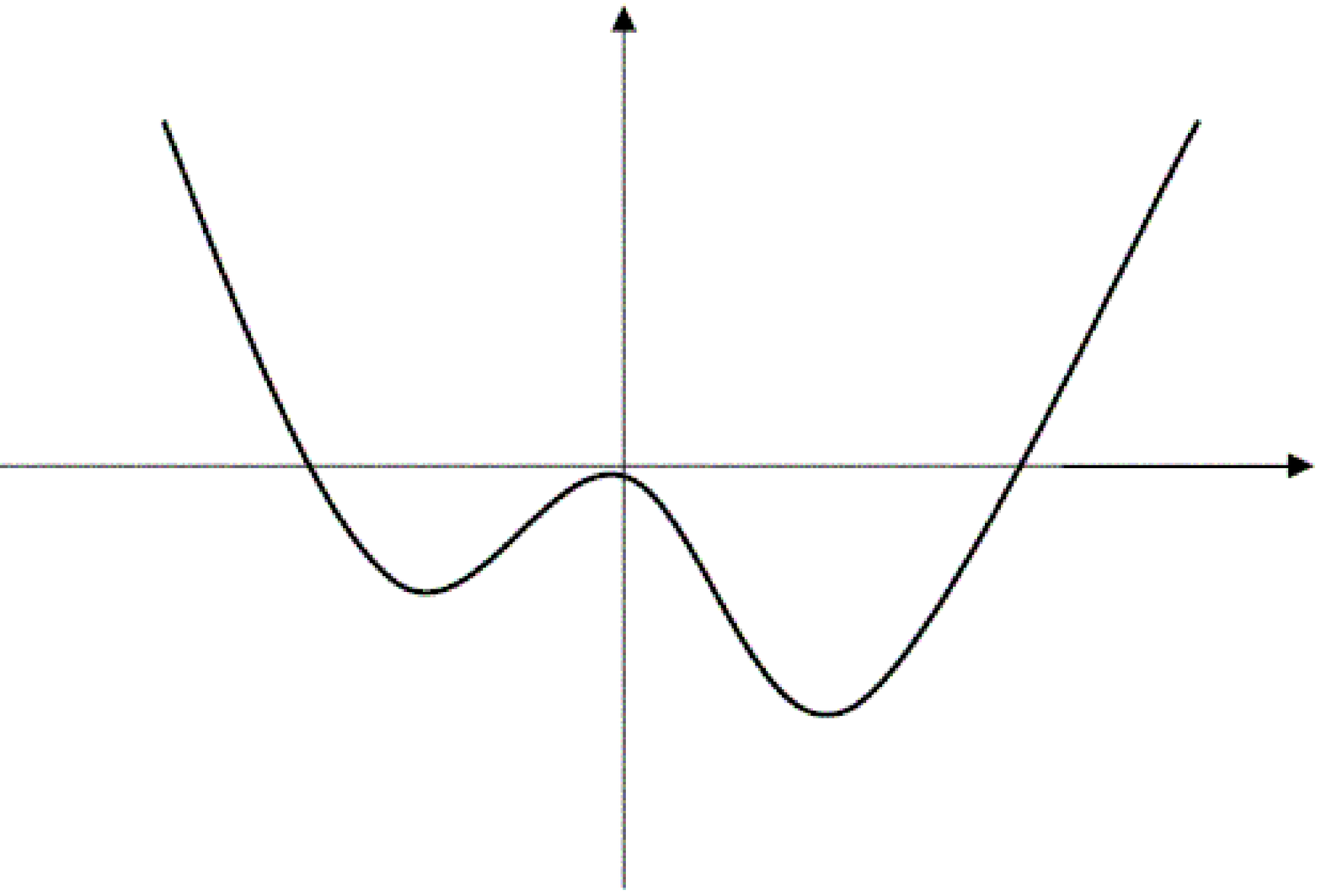}\includegraphics{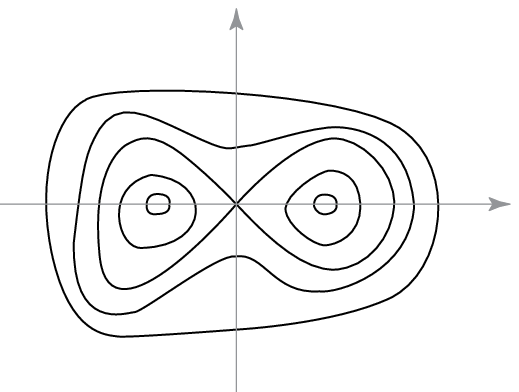}
\par\end{center}

\selectlanguage{english}%
\begin{center}
\textit{Fig. 1. \& 2. On figure 1 (left) the potential function $V(x).$
On figure 2 (right) the associated foliation in the phase plane.}
\par\end{center}

\vspace{0.25cm}

\subsection{Spectrum near the singularity}

In \textbf{{[}Co-Pa1{]}},\textbf{ {[}Co-Pa2{]}} and\textbf{ {[}Co-Pa3{]},}
Y. Colin de Verdi\`{e}re and B. Parisse gives an implicit singular Bohr-Sommerfeld
rules. An explicit description is given in \textbf{{[}Lab3{]}}. Record
here this\textbf{ }description. We take the presentation from the
paper \textbf{{[}Lab3{]} }and refer to this paper for more details.

\subsubsection{Some notations (see {[}Lab3{]})}

In \textbf{{[}Co-Pa1{]}},\textbf{ {[}Co-Pa2{]}},\textbf{ {[}Co-Pa3{]}}
and \textbf{{[}Lab3{]}} we have the smooth function $E\mapsto\varepsilon(E)$
defined for $|E|\leq\delta$, where $\delta$ is a real constant and
not depends to $h$; (this function is linked to a normal form \textbf{{[}CLP{]},
{[}Co-Pa1{]}}) defined by 

\[
\varepsilon(E)=\sum_{j=0}^{+\infty}\varepsilon_{j}(E)h^{j}\]
We have (see \textbf{{[}Co-Pa1{]}},\textbf{ {[}Co-Pa2{]}},\textbf{
{[}Co-Pa3{]}} and \textbf{{[}Lab3{]}}) the equality $\varepsilon_{0}(0)=0$
and $\varepsilon_{0}^{\prime}(0)=1/\sqrt{-V^{''}(0)}$. Hence, if
we use the Taylor formula on the smooth function $\varepsilon_{0}$
; for all $E\in[-\delta,\delta]$ we get\[
\varepsilon(E)=\frac{E}{\sqrt{-V^{''}(0)}}+O(E^{2})+\sum_{j=1}^{+\infty}\varepsilon_{j}(E)h^{j}.\]
So, for $\lambda\in[-1,1]$; and $h$ small enough (for have $[-h,h]\subset[-\delta,\delta])$
we get 

\[
\varepsilon(\lambda h)=\frac{\lambda h}{\sqrt{-V^{''}(0)}}+O(h^{2})+\sum_{j=1}^{+\infty}\varepsilon_{j}(\lambda h)h^{j}.\]
In the papers\textbf{ {[}Co-Pa1{]}},\textbf{ {[}Co-Pa2{]}},\textbf{
{[}Co-Pa3{]}} and \textbf{{[}Lab3{]}} we also use two smooth functions
$S^{+}$ and $S^{-}$ :

\[
S^{+/-}(E)={\displaystyle \sum_{j=0}^{+\infty}S_{j}^{+/-}(E)h^{j};}\]
where the functions are \foreignlanguage{english}{$E\rightarrow S_{j}^{+/-}(E)$
are }$\mathcal{C}^{\infty}$ smooth. This functions are usually called
singular actions. Let us also denote $\theta_{+/-}(E):=S^{+/-}(E)/h.$
This functions have a holonomy interpretation. In \textbf{{[}Lab3{]}}
we consider the functions $E\mapsto F_{h}(E)$ and $E\mapsto G_{h}(E)$
defined on the compact set $\left[-\delta,\delta\right]$ by : \textit{\[
F_{h}(E):=-\frac{\theta_{+}(E)+\theta_{-}(E)}{2}+\frac{\pi}{2}+\frac{\varepsilon(E)}{h}\ln(h)+\arg\left(\Gamma\left(\frac{1}{2}+i\frac{\varepsilon(E)}{h}\right)\right);\]
}$\Gamma$ is the Gamma function, and by \textit{\[
G_{h}(E):=\frac{\theta_{+}(E)-\theta_{-}(E)}{2}.\]
}On the compact set $\left[-1,1\right]$, let us consider the functions
$\lambda\mapsto f_{h}(\lambda)$ and $\lambda\mapsto g_{h}(\lambda)$
defined by : \textit{\[
f_{h}(\lambda):=F_{h}(\lambda h);\;\, g_{h}(\lambda):=G_{h}(\lambda h).\]
}For finish, let us consider the functions $\lambda\mapsto\mathcal{Y}_{h}(\lambda)$
and $\lambda\mapsto\mathcal{Z}_{h}(\lambda)$ defined on the compact
set $\left[-1,1\right]$ by \[
\mathcal{Y}_{h}(\lambda):=f_{h}(\lambda)-\arccos\left(\frac{\cos\left(g_{h}(\lambda)\right)}{\sqrt{1+\exp\left(2\pi\varepsilon(\lambda h)/h\right)}}\right);\]
\[
\mathcal{Z}_{h}(\lambda):=f_{h}(\lambda)+\arccos\left(\frac{\cos\left(g_{h}(\lambda)\right)}{\sqrt{1+\exp\left(2\pi\varepsilon(\lambda h)/h\right)}}\right).\]
We have the following result (see \textbf{{[}Lab3{]}}) :
\begin{prop}
For $h$ small enough, the function $\mathcal{Y}_{h}$ (resp. $\mathcal{Z}_{h}$)
is a bijection from $\left[-1,1\right]$ onto $\mathcal{Y}_{h}\left(\left[-1,1\right]\right)$
(resp. onto $\mathcal{Z}_{h}\left(\left[-1,1\right]\right)$). Moreover
on the compact set $\left[-1,1\right]$ we have\[
\mathcal{Y}_{h}^{\prime}(\lambda)=\frac{\ln(h)}{\sqrt{-V^{''}(0)}}+O(1).\]
Similary for the function $\mathcal{Z}_{h}$.
\end{prop}
Since the functions $\mathcal{Y}_{h}$ et $\mathcal{Z}_{h}$ are bijectives,
we can consider \[
\mathcal{A}_{h}:=\mathcal{Y}_{h}^{-1}\,:\,\mathcal{Y}_{h}\left([-1,1]\right)\rightarrow[-1,1];\]
\[
\mathcal{B}_{h}:=\mathcal{Z}_{h}^{-1}\,:\,\mathcal{Z}_{h}\left([-1,1]\right)\rightarrow[-1,1].\]

\begin{notation}
Let us denotes :\[
I_{h}:=\left\{ k\in\mathbb{Z},\,2\pi k\in\mathcal{Y}_{h}\left([-1,1]\right)\right\} =\frac{\mathcal{Y}_{h}\left([-1,1]\right)}{2\pi}\cap\mathbb{Z};\]
\[
J_{h}:=\left\{ \ell\in\mathbb{Z},\,2\pi\ell\in\mathcal{Z}_{h}\left([-1,1]\right)\right\} =\frac{\mathcal{Z}_{h}\left([-1,1]\right)}{2\pi}\cap\mathbb{Z}.\]

\end{notation}

\subsubsection{The main theorem}

The main theorem of the paper \textbf{{[}Lab3{]}} is the the following
:
\begin{thm}
\textbf{{[}Lab3{]}}. The semi-classical spectrum of ${\displaystyle P_{h}}$
in the compact set $\left[-\sqrt{h},\sqrt{h}\right]$ is the disjoint
union \[
\left(\alpha_{k}(h)\right)_{k\in\mathbf{I}_{h}}\bigsqcup\left(\beta_{\ell}(h)\right)_{\ell\in\mathbf{J}_{h}}\]
of two families $\left(\alpha_{k}(h)\right)_{k}$ and $\left(\beta_{\ell}(h)\right)_{\ell}$
such that $\alpha_{k}(h):=h\mathcal{A}_{h}(2\pi k)\in\mathbb{R},\,\beta_{\ell}(h):=h\mathcal{B}_{h}(2\pi\ell)\in\mathbb{R}$.
The functions $\mathcal{A}_{h}$ and $\mathcal{B}_{h}$ are $\mathcal{C}^{\infty}$
smooth. The families $\left(\alpha_{k}(h)\right)_{k}$, $\left(\beta_{\ell}(h)\right)_{\ell}$
are strictly non-increasing and : \[
\beta_{k+1}(h)<\alpha_{k}(h)<\beta_{k}(h)<\alpha_{k-1}(h).\]
Moreover, the spectral gap is of order $O(h/\left|\ln(h)\right|)$;
e.g : there exists $C,C^{\prime}>0$ such that : \[
\frac{Ch}{\left|\ln(h)\right|}\leq\left|\alpha_{k+1}(h)-\alpha_{k}(h)\right|\leq\frac{C^{\prime}h}{\left|\ln(h)\right|},\,\frac{Ch}{\left|\ln(h)\right|}\leq\left|\beta_{k+1}(h)-\beta_{k}(h)\right|\leq\frac{C^{\prime}h}{\left|\ln(h)\right|}.\]
\end{thm}
\begin{cor}
The number of eigenvalues in the compact set $\left[-\sqrt{h},\sqrt{h}\right]$
is of order $\left|\ln(h)\right|/\sqrt{h}$.
\end{cor}
In this paper, for technical reason, we just use the shape of the
spectrum in the compact set $\left[-h,h\right]\subset\left[-\sqrt{h},\sqrt{h}\right]$.
More precisely we choose an initial wave packet $\psi_{0}$ localized
in a compact set of size $h$. Indeed, the estimates in lemma 5.1
and 6.1 are very difficult to enable in the compact set $\left[-\sqrt{h},\sqrt{h}\right]$.
\begin{notation}
Let us denote by $\Theta_{h}$ the set of index for eigenvalues into
the compact set $[-h,h]$; e.g \[
\Theta_{h}:=\left\{ n\in\mathbb{N},\,\lambda_{n}(h)\in[-h,h]\right\} .\]

\end{notation}

\section{Some preliminaries}

In this section, we define partials autocorrelations functions, next
we define an initial vector $\psi_{0}$. And for finish, we introduce
the set $\Delta\subset\mathbb{N}$, this set is useful for making
estimates in sections 5 and 6.

\subsection{Partials autocorrelations functions }

Now, let be a initial vector $\psi_{0}={\displaystyle \sum_{n\in\mathbb{N}}c_{n}e_{n}}$
, then we have for all $t\geq0$

\[
\mathbf{r}(t)={\displaystyle \sum_{n\in\Theta_{h}}\left|c_{n}\right|^{2}e^{-i\frac{t}{h}\lambda_{n}(h)}+\sum_{n\in\mathbb{N}-\Theta_{h}}\left|c_{n}\right|^{2}e^{-i\frac{t}{h}\lambda_{n}(h)}}.\]
After, we let us consider an initial vector localized near the singularity
in a set of size $h$; hence the second series \[
\sum_{n\in\mathbb{N}-\Theta_{h}}\left|c_{n}\right|^{2}e^{-i\frac{t}{h}\lambda_{n}(h)}\]
will be equal to $O\left|\frac{1}{\ln(h)^{\infty}}\right|.$ Next,
the idea is to use the families $\left(\alpha_{k}(h)\right)_{k}$
and $\left(\beta_{\ell}(h)\right)_{\ell}$ from the theorem above
: as the eigenvalues of ${\displaystyle P_{h}}$ in the compact set
$[-h,h]$ are distinct; for all index $k\in\mathbf{I}_{h}$ and $\ell\in\mathbf{J}_{h}$
there exists a unique pair $\left(\sigma_{1}(k),\sigma_{2}(\ell)\right)\in\Theta_{h}^{2}$
such that :\[
\left\{ \begin{array}{cc}
h\mathcal{A}_{h}(2\pi k)=\lambda_{\sigma_{1}(k)}(h)\\
\\h\mathcal{B}_{h}(2\pi\ell)=\lambda_{\sigma_{2}(\ell)}(h).\end{array}\right.\]
So, we can consider the applications :\[
\sigma_{1}\,:\left\{ \begin{array}{cc}
\mathbf{I}_{h}\rightarrow\Theta_{h}\\
\\k\mapsto\sigma_{1}(k)\end{array}\right.\;\;\textrm{and}\;\;\sigma_{2}\,:\left\{ \begin{array}{cc}
\mathbf{J}_{h}\rightarrow\Theta_{h}\\
\\\ell\mapsto\sigma_{2}(\ell).\end{array}\right.\]
Clearly, this applications are injectives; hence $\mathbf{I}_{h}$
is isomorphic to $\sigma_{1}(\mathbf{I}_{h})$ (similary for $\mathbf{J}_{h}$
and $\sigma_{2}(\mathbf{J}_{h})$). Moreover we have \[
\Theta_{h}=\sigma_{1}(\mathbf{I}_{h})\bigsqcup\sigma_{2}(\mathbf{J}_{h}).\]
So, we get the equality : 

\[
{\displaystyle \mathbf{r}(t)=\sum_{k\in\sigma_{1}(\mathbf{I}_{h})}\left|c_{k}\right|^{2}e^{-i\frac{t}{h}\lambda_{k}(h)}}+\sum_{\ell\in\sigma_{2}(\mathbf{J}_{h})}\left|c_{\ell}\right|^{2}e^{-i\frac{t}{h}\lambda_{\ell}(h)}\]
\[
+\sum_{n\in\mathbb{N}-\Theta_{h}}\left|c_{n}\right|^{2}e^{-i\frac{t}{h}\lambda_{n}(h)}\]
\[
{\displaystyle =\sum_{k\in\mathbf{I}_{h}}\left|c_{\sigma_{1}(k)}\right|^{2}e^{-it\mathcal{A}_{h}(2\pi k)}}+\sum_{\ell\in\mathbf{J}_{h}}\left|c_{\sigma_{2}(\ell)}\right|^{2}e^{-it\mathcal{B}_{h}(2\pi\ell)}\]

\[
+\sum_{n\in\mathbb{N}-\Theta_{h}}\left|c_{n}\right|^{2}e^{-i\frac{t}{h}\lambda_{n}(h)}.\]
Now, our goal here is to study the series $t\mapsto\mathbf{a}$ et
$t\mapsto\mathbf{b}(t)$ :
\begin{defn}
Let us consider partial autocorrelation functions : \[
{\displaystyle \mathbf{a}(t):=\sum_{k\in\mathbf{I}_{h}}\left|c_{\sigma_{1}(k)}\right|^{2}e^{-it\mathcal{A}_{h}(2\pi k)},\;{\displaystyle \mathbf{b}(t):=\sum_{\ell\in\mathbf{J}_{h}}\left|c_{\sigma_{2}(\ell)}\right|^{2}e^{-it\mathcal{B}_{h}(2\pi\ell)}}.}\]

\end{defn}

\subsection{Choice of initial state}

\subsubsection{Prologue}

Let us define an initial vector $\psi_{0}={\displaystyle \sum_{n\in\mathbb{N}}c_{n}e_{n}}$
localized near the real number $hE$ (where $E\in\left[-1,1\right]$). 
\begin{defn}
Let us consider the quantum integers $n_{0}=n_{0}(h,E)$ and $m_{0}=m_{0}(h,E)$
defined by \[
n_{0}:=\textrm{arg}\min_{n\in\mathbb{\mathbf{I}}_{h}}\left|h\mathcal{A}_{h}(2\pi n)-hE\right|\;\;\textrm{and}\;\; m_{0}:=\textrm{arg}\min_{m\in\mathbb{\mathbf{I}}_{h}}\left|h\mathcal{B}_{h}(2\pi m)-hE\right|.\]
\end{defn}
\begin{rem}
Without loss of generality, we may suppose $n_{0}$ and $m_{0}$ are
unique.
\end{rem}
The integer $n_{0}$ (resp. $m_{0}$) is the index of the eigenvalues
from the family $\left(\alpha_{k}(h)\right)_{k}$ (resp. $\left(\beta_{\ell}(h)\right)_{\ell}$)
the closest to the real number $hE$. As the spectral gap is of order
$O(h/\left|\ln(h)\right|)$ there exists $C>0$ such that :

\[
\left|h\mathcal{A}_{h}(2\pi n_{0})-hE\right|\leq\frac{Ch}{\left|\ln h\right|},\;\left|h\mathcal{B}_{h}(2\pi m_{0})-hE\right|\leq\frac{Ch}{\left|\ln h\right|}.\]
So, we get :
\begin{lem}
As $h\rightarrow0,$ we have \[
n_{0}\sim\frac{N}{h}\;\textrm{and}\; m_{0}\sim\frac{M}{h}\]
where $N$ and $M$ non-null real numbers.\end{lem}
\begin{proof}
We give the proof for the integer $n_{0}$. As\[
\left|\mathcal{A}_{h}(2\pi n_{0})-E\right|\leq\frac{C}{\left|\ln h\right|};\]
e.g.\[
\mathcal{A}_{h}(2\pi n_{0})=E+O\left(\frac{1}{\left|\ln h\right|}\right),\]
we get\[
2\pi n_{0}=\mathcal{Y}_{h}\left(E+O\left(\frac{1}{\left|\ln h\right|}\right)\right)\]
\[
=f_{h}\left(E+O\left(\frac{1}{\left|\ln h\right|}\right)\right)-\arccos\left(\frac{\cos\left(g_{h}\left(E+O\left(\frac{1}{\left|\ln h\right|}\right)\right)\right)}{\sqrt{1+\exp\left(2\pi\varepsilon\left(hE+O\left(\frac{h}{\left|\ln h\right|}\right)\right)/h\right)}}\right).\]
By definition :\[
f_{h}\left(E+O\left(\frac{1}{\left|\ln h\right|}\right)\right)=F_{h}\left(hE+O\left(\frac{h}{\left|\ln h\right|}\right)\right)\]
\textit{\[
=-\frac{\theta_{+}\left(hE+O\left(h/\left|\ln(h)\right|\right)\right)+\theta_{-}\left(hE+O\left(h/\left|\ln(h)\right|\right)\right)}{2}+\frac{\pi}{2}\]
}\[
+\frac{\varepsilon\left(hE+O\left(h/\left|\ln(h)\right|\right)\right)}{h}\ln(h)+\arg\left(\Gamma\left(\frac{1}{2}+i\frac{\varepsilon\left(hE+O\left(h/\left|\ln(h)\right|\right)\right)}{h}\right)\right).\]
Hence, if we multiply $2\pi n_{0}$ by $h$ we obtain\[
2\pi n_{0}h=-h\frac{\theta_{+}\left(hE+O\left(h/\left|\ln(h)\right|\right)\right)+\theta_{-}\left(hE+O\left(h/\left|\ln(h)\right|\right)\right)}{2}\]
\textit{\[
+\frac{\pi}{2}h+\varepsilon\left(hE+O\left(h/\left|\ln(h)\right|\right)\right)\ln(h)+h\arg\left(\Gamma\left(\frac{1}{2}+i\frac{\varepsilon\left(hE+O\left(h/\left|\ln(h)\right|\right)\right)}{h}\right)\right)\]
}\[
-h\arccos\left(\frac{\cos\left(g_{h}\left(E+O\left(\frac{1}{\left|\ln h\right|}\right)\right)\right)}{\sqrt{1+\exp\left(2\pi\varepsilon\left(hE+O\left(\frac{h}{\left|\ln h\right|}\right)\right)/h\right)}}\right).\]
Let us evaluate this five terms. As the function \[
E\mapsto-\frac{\theta_{+}\left(hE+O\left(h/\left|\ln(h)\right|\right)\right)+\theta_{-}\left(hE+O\left(h/\left|\ln(h)\right|\right)\right)}{2}\]
admit a asymptotic expansion in power of $h$ from $-1$ to $+\infty$,
we have \[
-h\frac{\theta_{+}\left(hE+O\left(h/\left|\ln(h)\right|\right)\right)+\theta_{-}\left(hE+O\left(h/\left|\ln(h)\right|\right)\right)}{2}=O(1);\]
\foreignlanguage{english}{more precisely, for $h\rightarrow0$ \[
-h\frac{\theta_{+}\left(hE+O\left(h/\left|\ln(h)\right|\right)\right)+\theta_{-}\left(hE+O\left(h/\left|\ln(h)\right|\right)\right)}{2}\rightarrow-\frac{1}{2}\left(S_{-1}^{+}(0)+S_{-1}^{-}(0)\right)\neq0.\]
}Next :\[
\varepsilon\left(hE+O\left(h/\left|\ln(h)\right|\right)\right)\ln(h)=\left[\frac{hE+O\left(h/\left|\ln(h)\right|\right)}{\sqrt{-V^{''}(0)}}+O(h^{2})\right]\ln(h)\rightarrow0.\]
For finish, as\[
\arg\left(\Gamma\left(\frac{1}{2}+i\frac{\varepsilon\left(hE+O\left(h/\left|\ln(h)\right|\right)\right)}{h}\right)\right)=O(1)\]
and\[
\arccos\left(\frac{\cos\left(g_{h}\left(E+O\left(\frac{1}{\left|\ln h\right|}\right)\right)\right)}{\sqrt{1+\exp\left(2\pi\varepsilon\left(hE+O\left(\frac{h}{\left|\ln h\right|}\right)\right)/h\right)}}\right)=O(1);\]
we get, for $h\rightarrow0$ :\[
h\arg\left(\Gamma\left(\frac{1}{2}+i\frac{\varepsilon\left(hE+O\left(h/\left|\ln(h)\right|\right)\right)}{h}\right)\right)\rightarrow0\]
and\[
h\arccos\left(\frac{\cos\left(g_{h}\left(E+O\left(\frac{1}{\left|\ln h\right|}\right)\right)\right)}{\sqrt{1+\exp\left(2\pi\varepsilon\left(hE+O\left(\frac{h}{\left|\ln h\right|}\right)\right)/h\right)}}\right)\rightarrow0.\]
Hence, \[
\lim_{h\rightarrow0}2\pi n_{0}h=-\frac{1}{2}\left(S_{-1}^{+}(0)+S_{-1}^{-}(0)\right);\]
so we prove the lemma.
\end{proof}

\subsubsection{First definition (non definitive)}

For technical reason, let us introduce the following sequence :
\begin{defn}
Let us consider : \textbf{\[
\mu_{n,0}:=\left\{ \begin{array}{cc}
h\mathcal{A}_{h}(2\pi n_{0})\;\;\mathbf{\textrm{if }}n\in\sigma_{1}(\mathbf{I}_{h})\\
h\mathcal{B}_{h}(2\pi m_{0})\;\;\mathbf{\textrm{if }}n\in\sigma_{2}(\mathbf{J}_{h}).\\
0\;\;\mathbf{\textrm{if }}n\in\mathbb{N}-\Theta_{h}.\end{array}\right.\]
}
\end{defn}
For $n\in\Theta_{h}=\sigma_{1}(\mathbf{I}_{h})\bigsqcup\sigma_{2}(\mathbf{J}_{h})$,
the sequence $\left(\mu_{n,0}\right)_{n}$ take only two values : 
\begin{itemize}
\item the closest eigenvalue from the family $\left(\alpha_{k}(h)\right)_{k}$
to the real number $hE$; 
\item or the closest eigenvalue from the family $\left(\beta_{\ell}(h)\right)_{\ell}$
to the real number $hE$.
\end{itemize}
Now, we can give the first (non definitive) definition of our initial
state.
\begin{defn}
Let us consider the sequence $\left(c_{n}(h)\right)_{n\in\mathbb{Z}}$
defined by :\[
c_{n}:=c_{n}(h)=K_{h}\chi\left(\frac{\mu_{n}-\mu_{n,0}}{h^{\alpha^{\prime}}}\left|\ln h\right|^{\gamma^{\prime}}\right)\chi_{0}\left(\frac{\mu_{n}-\mu_{n,0}}{2h}\right),\, n\in\mathbb{Z}\]
where $\chi\in\mathcal{S}(\mathbb{R})$ , non null, non-negative and
even, $\alpha^{\prime}$ and $\gamma^{\prime}$ are two reals numbers;
$\chi_{0}\in\mathcal{D}(\mathbb{R})$ such that $\chi_{0}\equiv1$
on the set $\left]-1,1\right[$ and $\textrm{supp}(\chi_{0})\subset\left[-1,1\right]$.
We also denote\[
K_{h}:=\left\Vert \chi\left(\frac{\mu_{n}-\mu_{n,0}}{h^{\alpha^{\prime}}}\left|\ln h\right|^{\gamma^{\prime}}\right)\chi_{0}\left(\frac{\mu_{n}-\mu_{n,0}}{2h}\right)\right\Vert _{\ell^{2}(\mathbb{N})}.\]

\end{defn}
Let us detail this choice :
\begin{enumerate}
\item The term \foreignlanguage{english}{$\chi\left(\frac{\mu_{n}-\mu_{n,0}}{h^{\alpha^{\prime}}}\left|\ln h\right|^{\gamma^{\prime}}\right)$
localizes around the energy level $hE$ (for technical reason we localize
around the }closest eigenvalues\foreignlanguage{english}{ to $E_{h}$.}
\selectlanguage{english}%
\item Constants\foreignlanguage{english}{ $\alpha^{\prime}$ et $\gamma^{\prime}$
are coefficients for dilate the function $\chi$.}
\selectlanguage{english}%
\item The function $\chi_{0}$ is a cut-off for eigenvalues out of the compact
set $\left[-h,h\right]$.
\end{enumerate}

\subsubsection{Choice of parameter $\alpha^{\prime}$ and $\gamma^{\prime}$}

The only way to have a localization for initial state larger than
the spectral gap, and in the same time, to have a localization in
the compact set $\left[-h,h\right]$; e.g to have in the same time
:\[
\frac{h}{\left|\ln h\right|}\ll\frac{h^{\alpha^{\prime}}}{\left|\ln h\right|^{\gamma^{\prime}}}\leq h;\]
is to take\[
\alpha^{\prime}=1;\;0\leq\gamma^{\prime}<1.\]

\begin{rem}
The larger choice for $h^{\alpha^{\prime}}/\left|\ln h\right|^{\gamma^{\prime}}=h/\left|\ln h\right|^{\gamma^{\prime}}$
is to take $\gamma^{\prime}=0$.
\end{rem}
Under this assumptions, we can forget the function $\chi_{0}$ and
we get :

\subsubsection{Second definition}
\begin{defn}
Let us consider the sequence $\left(c_{n}(h)\right)_{n\in\mathbb{Z}}$
defined by :\[
c_{n}:=K_{h}\chi\left(\frac{\mu_{n}-\mu_{n,0}}{h}\left|\ln h\right|^{\gamma^{\prime}}\right),\, n\in\mathbb{Z}\]
where $\chi\in\mathcal{S}(\mathbb{R})$ , non null, non-negative and
even; $0\leq\gamma^{\prime}<1$; and\[
K_{h}:=\left\Vert \chi\left(\frac{\mu_{n}-\mu_{n,0}}{h}\left|\ln h\right|^{\gamma^{\prime}}\right)\right\Vert _{\ell^{2}(\mathbb{N})}.\]

\end{defn}
Now, in the partial autocorrelations functions $\mathbf{a}(t)$ and
$\mathbf{b}(t)$ appears the sequences $\left(c_{\sigma_{1}(k)}\right)_{k}$and
$\left(c_{\sigma_{2}(\ell)}\right)_{\ell}$; for simplicity let us
consider the following sequences :
\begin{notation}
Let us denote\[
a_{n}:=c_{\sigma_{1}(n)}=K_{h}\chi\left(\left(\mathcal{A}_{h}(2\pi n)-\mathcal{A}_{h}(2\pi n_{0})\right)\left|\ln h\right|^{\gamma^{\prime}}\right);\]
\[
b_{m}:=c_{\sigma_{2}(m)}=K_{h}\chi\left(\left(\mathcal{B}_{h}(2\pi m)-\mathcal{B}_{h}(2\pi m_{0})\right)\left|\ln h\right|^{\gamma^{\prime}}\right).\]

\end{notation}

\subsubsection{Last definition}

By Lagrange's theorem there exists a real number $\zeta=\zeta(n,h,E)\in\mathcal{Y}_{h}\left([-1,1]\right)$
such that\[
\mathcal{A}_{h}(2\pi n)-\mathcal{A}_{h}(2\pi n_{0})=\mathcal{A}_{h}^{\prime}(\zeta)2\pi(n-n_{0});\]
since (see proposition 3.2)

\selectlanguage{english}%
\[
\mathcal{A}_{h}^{\prime}(\zeta)=\frac{2\pi}{\frac{\left|\ln(h)\right|}{\sqrt{-V^{''}(0)}}+O(1)};\]
we get\foreignlanguage{english}{\[
\mathcal{A}_{h}(2\pi n)-\mathcal{A}_{h}(2\pi n_{0})=\frac{2\pi(n-n_{0})\sqrt{-V^{''}(0)}}{\left|\ln h\right|}\left(1+O\left|\frac{1}{\left|\ln h\right|}\right|\right).\]
}
\selectlanguage{english}%
\begin{defn}
Let us consider the sequence $\left(a_{n}\right)_{n\in\mathbb{Z}}=\left(a_{n}(h)\right)_{n\in\mathbb{Z}}$
defined by :\[
a_{n}:=K_{h}\chi\left(\frac{n-n_{0}}{\left|\ln h\right|^{1-\gamma^{\prime}}}\right),\, n\in\mathbb{Z}\]
where $\chi\in\mathcal{S}(\mathbb{R})$ , non null, non-negative and
even; $0\leq\gamma^{\prime}<1$; and\[
K_{h}:=\left\Vert \chi\left(\frac{n-n_{0}}{\left|\ln h\right|^{1-\gamma^{\prime}}}\right)\right\Vert _{\ell^{2}(\mathbb{N})}.\]
So, clearly the sequence $\left(a_{n}\right)_{n}\in\ell^{2}(\mathbb{Z}).$
Now, let us evaluate the constant of normalization $K_{h}$. Start
by the :\end{defn}
\begin{lem}
For a function $\varphi\in\mathcal{S}(\mathbb{R})$ and $\varepsilon\in\left]0,1\right]$;
we have :\[
\sum_{\ell\in\mathbb{\mathbb{Z}},\,\left|\ell\right|\geq1}\left|\varphi\left(\frac{\ell}{\varepsilon}\right)\right|=O(\varepsilon^{\infty}).\]
\end{lem}
\begin{proof}
The starting point is the following remark : for all function \foreignlanguage{english}{$\psi\in\mathcal{S}(\mathbb{R})$}
and for all $\varepsilon\in\left]0,1\right]$ we have \[
\sum_{\ell\in\mathbb{\mathbb{Z}},\,\left|\ell\right|\geq1}\left|\psi\left(\frac{\ell}{\varepsilon}\right)\right|=O(1).\]
Indeed, for all $\ell\in\mathbb{\mathbb{Z}},\,\left|\ell\right|\geq1$\[
\left|\psi\left(\frac{\ell}{\varepsilon}\right)\right|\leq\frac{M}{1+\left(\frac{\ell}{\varepsilon}\right)^{2}}=\frac{\varepsilon^{2}M}{\varepsilon^{2}+\ell^{2}}\]
\[
\leq\frac{M}{\ell^{2}}\]
and since \[
\sum_{\ell\in\mathbb{\mathbb{Z}},\,\left|\ell\right|\geq1}\frac{M}{\ell^{2}}<+\infty.\]
we conclude.

Next, we apply this to the function $\psi(x):=x^{2N}\varphi(x)$,
where $N\in\mathbb{N}$; and we get for all $N\geq1$ :\[
\sum_{\ell\in\mathbb{\mathbb{Z}},\,\left|\ell\right|\geq1}\left|\varphi\left(\frac{\ell}{\varepsilon}\right)\right|\leq\varepsilon^{2N}\sum_{\ell\in\mathbb{\mathbb{Z}}}\left|\psi\left(\frac{\ell}{\varepsilon}\right)\right|=O(\varepsilon^{2N}).\]
So we prove the lemma.\end{proof}
\begin{thm}
We have\[
K_{h}=\frac{1}{\sqrt{\mathfrak{F}\left(\chi^{2}\right)(0)}\left|\ln h\right|^{\frac{1-\gamma^{\prime}}{2}}}+O\left|\frac{1}{\ln(h)^{\infty}}\right|;\]
so\[
\left\Vert a_{n}\right\Vert _{\ell^{2}(\mathbb{N})}=1+O\left|\frac{1}{\ln(h)^{\infty}}\right|.\]
\end{thm}
\begin{proof}
By Poisson formula and the lemma above we get the equality\textit{\[
\sum_{n\in\mathbb{Z}}\chi^{2}\left(\frac{n-n_{0}}{\left|\ln h\right|^{1-\gamma^{\prime}}}\right)=\left|\ln h\right|^{1-\gamma^{\prime}}\sum_{\ell\in\mathbb{Z}}\mathfrak{F}\left(\chi^{2}\right)\left(-\ell\left|\ln h\right|^{1-\gamma^{\prime}}\right)\]
\[
=\left|\ln h\right|^{1-\gamma^{\prime}}\left[\mathfrak{F}\left(\chi^{2}\right)(0)+O\left|\frac{1}{\ln(h)^{\infty}}\right|\right]\]
}Now, start with the equality \textit{\[
\sum_{n\in\mathbb{N}}\chi^{2}\left(\frac{n-n_{0}}{\left|\ln h\right|^{1-\gamma^{\prime}}}\right)=\sum_{n\in\mathbb{Z}}\chi^{2}\left(\frac{n-n_{0}}{\left|\ln h\right|^{1-\gamma^{\prime}}}\right)-\sum_{n=-\infty}^{-1}\chi^{2}\left(\frac{n-n_{0}}{\left|\ln h\right|^{1-\gamma^{\prime}}}\right)\]
\[
=\left|\ln h\right|^{1-\gamma^{\prime}}\mathfrak{F}\left(\chi^{2}\right)(0)+O\left|\frac{1}{\ln(h)^{\infty}}\right|-\sum_{n=-\infty}^{-1}\chi^{2}\left(\frac{n-n_{0}}{\left|\ln h\right|^{1-\gamma^{\prime}}}\right).\]
}Since \foreignlanguage{english}{the function $\mathfrak{F}\left(\chi^{2}\right)$
is} even\foreignlanguage{english}{ :}\[
\sum_{n=-\infty}^{-1}\chi^{2}\left(\frac{n-n_{0}}{\left|\ln h\right|^{1-\gamma^{\prime}}}\right)=\sum_{n=1}^{+\infty}\chi^{2}\left(\frac{n+n_{0}}{\left|\ln h\right|^{1-\gamma^{\prime}}}\right)\leq B_{k}\left|\ln h\right|^{(1-\gamma^{\prime})k}\sum_{n=1}^{+\infty}\frac{1}{\left|n+n_{0}\right|^{k}}.\]
And, \[
{\displaystyle B_{k}\left|\ln h\right|^{(1-\gamma^{\prime})k}\sum_{n=1}^{+\infty}\frac{1}{\left|n+n_{0}\right|^{k}}}\leq{\displaystyle B_{k}\left|\ln h\right|^{(1-\gamma^{\prime})k}}\int_{0}^{\infty}\frac{du}{(u+n_{0})^{k}}\]
\[
=\left|\ln h\right|^{(1-\gamma^{\prime})(k-1)}\frac{B_{k}}{(k-1)}\frac{1}{n_{0}^{k-1}}.\]
Now, since for $h\rightarrow0$ we have $n_{0}\sim\frac{N}{h}$ ,
we get, for $h\rightarrow0$\[
\left|\ln h\right|^{(1-\gamma^{\prime})k}\frac{B_{k}}{(k-1)}\frac{1}{n_{0}^{k-1}}\sim\frac{B_{k}}{N^{k-1}(k-1)}h^{k-1}\left|\ln h\right|^{(1-\gamma^{\prime})k}\]
consequently \foreignlanguage{english}{\[
\sum_{n=-\infty}^{-1}\chi^{2}\left(\frac{n-n_{0}}{\left|\ln h\right|^{1-\gamma^{\prime}}}\right)=O\left|\frac{1}{\ln(h)^{\infty}}\right|.\]
So, we proove that}\textit{\[
\left\Vert \chi\left(\frac{n-n_{0}}{\left|\ln h\right|^{1-\gamma^{\prime}}}\right)\right\Vert _{\ell^{2}(\mathbb{N})}^{2}=\left|\ln h\right|^{1-\gamma^{\prime}}\mathfrak{F}\left(\chi^{2}\right)(0)+O\left|\frac{1}{\ln(h)^{\infty}}\right|;\]
}hence $K_{h}=\frac{1}{\sqrt{\mathfrak{F}\left(\chi^{2}\right)(0)}\left|\ln h\right|^{\frac{1-\gamma^{\prime}}{2}}}+O\left|\frac{1}{\ln(h)^{\infty}}\right|.$
For finish, we have \textit{\[
\left\Vert a_{n}\right\Vert _{\ell^{2}(\mathbb{N})}^{2}=K_{h}^{2}\left|\ln h\right|^{1-\gamma^{\prime}}\left[\mathfrak{F}\left(\chi^{2}\right)(0)+O\left|\frac{1}{\ln(h)^{\infty}}\right|\right]=1+O\left|\frac{1}{\ln(h)^{\infty}}\right|.\]
}
\end{proof}

\subsection{The set $\Delta$}
\begin{defn}
Let us define the sets of integers $\Delta=\Delta(h,E)$ and $\Gamma=\Gamma(h,E)$
by :\[
\Delta:=\left\{ n\in\mathbb{N},\,\left|n-n_{0}\right|\leq\left|\ln(h)\right|^{\gamma}\right\} \subset\mathbb{N}\]
and\[
\Gamma:=\mathbb{N}-\Delta\]
where $\gamma$ is a real number such that $\gamma<1$ and $\gamma+\gamma^{\prime}>1$.\end{defn}
\begin{rem}
Since $\gamma<1$, we have $\left|\ln(h)\right|^{\gamma}<\left|\ln(h)\right|,$
hence for $h\rightarrow0$ we have\[
\textrm{Card}\left(\mathbf{I}_{h}\right)\thicksim\left|\ln(h)\right|.\]
So, for $h$ small enough we obtain $\Delta\subset\mathbf{I}_{h}$.
On the other hand, since $\gamma+\gamma^{\prime}>1$ we have $\left|\ln(h)\right|^{1-\gamma^{\prime}}\ll\left|\ln(h)\right|^{\gamma};$
this mean that the set $\Delta$ is larger than the localization of
initial state.\end{rem}
\begin{lem}
We have\[
{\displaystyle \sum_{n\in\Gamma}\left|a_{n}\right|^{2}}=O\left|\frac{1}{\ln(h)^{\infty}}\right|.\]
\end{lem}
\begin{proof}
The starting point is the following inequality :\[
{\displaystyle \sum_{n\in\Gamma}\left|a_{n}\right|^{2}}\leq{\displaystyle \sum_{n\in\mathbb{Z},\,\left|n-n_{0}\right|>h^{\delta-1}}\left|a_{n}\right|^{2}}.\]
Next, with the same argument as in the lemma 4.11, we show that for
all integer $N\geq1$\foreignlanguage{english}{\[
{\displaystyle \sum_{n\in\mathbb{Z}}\left(\frac{n-n_{0}}{h^{\delta^{\prime}-1}}\right)^{2N}}\left|a_{n}\right|^{2}=O(1).\]
}Without loss generality, we may suppose that $n_{0}=0$. Next we
write :\foreignlanguage{english}{\[
{\displaystyle \sum_{n\in\mathbb{Z},\,\left|n\right|>h^{\delta-1}}\left|a_{n}\right|^{2}}=h^{2N(\delta^{\prime}-1)}\sum_{n\in\mathbb{Z},\,\left|n\right|>h^{\delta-1}}\left|a_{n}\right|^{2}\left(\frac{n}{h^{\delta^{\prime}-1}}\right)^{2N}\frac{1}{n^{2N}}\]
\[
\leq\frac{h^{2N(\delta^{\prime}-1)}}{h^{2N(\delta-1)}}\sum_{n\in\mathbb{Z}}\left|a_{n}\right|^{2}\left(\frac{n}{h^{\delta^{\prime}-1}}\right)^{2N}=O\left(h^{2N(\delta^{\prime}-\delta)}\right).\]
}Since $\delta^{\prime}>\delta,$ we obtain ${\displaystyle \sum_{n\in\Gamma}\left|a_{n}\right|^{2}}=O(h^{\infty}).$
\end{proof}

\section{Order 1 approximation : hyperbolic period}

\subsection{Introduction}

The aim of this section is to study the series :\[
\mathbf{a}\,:\, t\mapsto{\displaystyle \sum_{n\in\mathbb{N}}\left|a_{n}\right|^{2}e^{-it\mathcal{A}_{h}(2\pi n)}.}\]
Unfortunately this function is too difficult to understand. So, in
this section we use an approximation of $t\mapsto\mathbf{a}(t)$,
the tricks is explained in section 2.3. Here, by a Taylor's formula
we get :\[
\mathbf{a}(t)={\displaystyle \sum_{n\in\mathbb{N}}\left|a_{n}\right|^{2}e^{-it\left(\mathcal{A}_{h}(2\pi n_{0})+\mathcal{A}_{h}^{\prime}(2\pi n_{0})2\pi(n-n_{0})+\mathcal{A}_{h}^{\prime\prime}(\zeta)2\pi^{2}(n-n_{0})^{2}\right)}}\]
where $\zeta=\zeta(n,h,E)\in\mathcal{Y}_{h}\left([-1,1]\right)$.

Next, we need the following lemma.
\begin{lem}
Uniformly, on the compact set $[-1,1]$ we have :\[
\lambda\mapsto\left(\mathcal{A}_{h}^{\prime\prime}\circ\mathcal{Y}_{h}\right)(\lambda)=O\left|\frac{1}{\ln(h)^{3}}\right|.\]
\end{lem}
\begin{proof}
Derivatives formulas gives for all $x\in\mathcal{Y}_{h}\left([-1,1]\right),$
the equality\[
\mathcal{A}_{h}^{\prime\prime}(x)=\frac{-\left(\mathcal{Y}_{h}^{\prime\prime}\circ\mathcal{A}_{h}\right)(x)}{\left(\mathcal{Y}_{h}^{\prime}\circ\mathcal{A}_{h}\right)^{3}(x)};\]
hence, for $\lambda\in[-1,1]$\[
\left(\mathcal{A}_{h}^{\prime\prime}\circ\mathcal{Y}_{h}\right)(\lambda)=\frac{-\mathcal{Y}_{h}^{\prime\prime}(\lambda)}{\left(\mathcal{Y}_{h}^{\prime}\right)^{3}(\lambda)}.\]
First, for all $\lambda\in\left[-1,1\right]$ \[
\mathcal{Y}_{h}^{\prime\prime}(\lambda)=-h^{2}\frac{\theta_{+}^{\prime\prime}(\lambda h)+\theta_{-}^{\prime\prime}(\lambda h)}{2}+h\varepsilon^{\prime\prime}(\lambda h)\ln(h)\]
\[
+\frac{\partial^{2}}{\partial\lambda^{2}}\left[\arg\left(\Gamma\left(\frac{1}{2}+i\frac{\varepsilon(\lambda h)}{h}\right)\right)\right]-\frac{\partial^{2}}{\partial\lambda^{2}}\left[\arccos\left(\frac{\cos\left(g_{h}(\lambda)\right)}{\sqrt{1+\exp\left(2\pi\varepsilon(\lambda h)/h\right)}}\right)\right].\]
And, for all $\lambda\in\left[-1,1\right]$ we have\[
-h^{2}\frac{\theta_{+}^{\prime\prime}(\lambda h)+\theta_{-}^{\prime}(\lambda h)}{2}=O(h)\]
and \[
h\varepsilon^{\prime\prime}(\lambda h)\ln(h)=O(h\ln(h)).\]
Next, we study the term $\frac{\partial^{2}}{\partial\lambda^{2}}\left[\arg\left(\Gamma\left(\frac{1}{2}+i\frac{\varepsilon(\lambda h)}{h}\right)\right)\right]$
: for all $\lambda\in[-1,1]$ we have :\[
\frac{\partial}{\partial\lambda}\left[\arg\left(\Gamma\left(\frac{1}{2}+i\frac{\varepsilon(\lambda h)}{h}\right)\right)\right]=\frac{\partial}{\partial\lambda}\left[\textrm{Im}\left(\ln\left(\Gamma\left(\frac{1}{2}+i\frac{\varepsilon(\lambda h)}{h}\right)\right)\right)\right]\]
\[
=\textrm{Im}\left[\frac{\partial}{\partial\lambda}\left(\ln\left(\Gamma\left(\frac{1}{2}+i\frac{\varepsilon(\lambda h)}{h}\right)\right)\right)\right]=\textrm{Im}\left[\frac{\Gamma^{\prime}\left(\frac{1}{2}+i\frac{\varepsilon(\lambda h)}{h}\right)i\varepsilon^{\prime}(\lambda h)}{\Gamma\left(\frac{1}{2}+i\frac{\varepsilon(\lambda h)}{h}\right)}\right]\]
\[
=\varepsilon^{\prime}(\lambda h)\textrm{Re}\left[\frac{\Gamma^{\prime}\left(\frac{1}{2}+i\frac{\varepsilon(\lambda h)}{h}\right)}{\Gamma\left(\frac{1}{2}+i\frac{\varepsilon(\lambda h)}{h}\right)}\right]=\varepsilon^{\prime}(\lambda h)\textrm{Re}\left(\Psi\left(\frac{1}{2}+i\frac{\varepsilon(\lambda h)}{h}\right)\right);\]
hence, for all $\lambda\in[-1,1]$ 

\[
\frac{\partial^{2}}{\partial\lambda^{2}}\left[\arg\left(\Gamma\left(\frac{1}{2}+i\frac{\varepsilon(\lambda h)}{h}\right)\right)\right]\]
\[
=\frac{\partial}{\partial\lambda}\left[\varepsilon^{\prime}(\lambda h)\textrm{Re}\left(\Psi\left(\frac{1}{2}+i\frac{\varepsilon(\lambda h)}{h}\right)\right)\right]\]
\[
=h\varepsilon^{\prime\prime}(\lambda h)\textrm{Re}\left(\Psi\left(\frac{1}{2}+i\frac{\varepsilon(\lambda h)}{h}\right)\right)-\left(\varepsilon^{\prime}(\lambda h)\right)^{2}\textrm{Im}\left(\Psi^{(1)}\left(\frac{1}{2}+i\frac{\varepsilon(\lambda h)}{h}\right)\right),\]
wher $\Psi^{(1)}$ is the first-derivative of di-Gamma function ($\Psi(z):=\frac{\Gamma^{\prime}(z)}{\Gamma(z)}$,
see\textbf{ {[}Ab-St{]}}). Hence the function\[
\lambda\mapsto\frac{\partial}{\partial\lambda}\left[\varepsilon^{\prime}(\lambda h)\textrm{Re}\left(\Psi\left(\frac{1}{2}+i\frac{\varepsilon(\lambda h)}{h}\right)\right)\right]\]
is equal to $O(1)$ on the compact set $\left[-1,1\right]$. For finish,
we estimate the term : \[
\lambda\mapsto\frac{\partial^{2}}{\partial\lambda^{2}}\left[\arccos\left(\frac{\cos\left(g_{h}(\lambda)\right)}{\sqrt{1+\exp\left(2\pi\varepsilon(\lambda h)/h\right)}}\right)\right];\]
for all $\lambda\in\left[-1,1\right]$ we obtain\[
\frac{\partial^{2}}{\partial\lambda^{2}}\left[\arccos\left(\frac{\cos\left(g_{h}(\lambda)\right)}{\sqrt{1+\exp\left(2\pi\varepsilon(\lambda h)/h\right)}}\right)\right]=\frac{N_{h}(\lambda)}{D_{h}(\lambda)}\]
where $N_{h}(\lambda)$ is a polynomial (with coefficients does not
depends on $h$) in the variables :\[
\left(\cos\left(g_{h}(\lambda)\right),\sin\left(g_{h}(\lambda)\right),\frac{\partial}{\partial\lambda}\left(g_{h}(\lambda)\right),\frac{\partial^{2}}{\partial\lambda^{2}}\left(g_{h}(\lambda)\right)\right.\]
\[
\left.\frac{\partial}{\partial\lambda}\left(\varepsilon(\lambda h)/h\right),\frac{\partial^{2}}{\partial\lambda^{2}}\left(\varepsilon(\lambda h)/h\right),\exp\left(2\pi\varepsilon(\lambda h)/h\right)\right).\]
Hence $\lambda\mapsto N_{h}(\lambda)=O(1)$ on the compact set $\left[-1,1\right].$
Next we have\[
D_{h}(\lambda)=\left(1+\exp\left(2\pi\varepsilon(\lambda h)/h\right)\right)^{2}\left(1+\exp\left(2\pi\varepsilon(\lambda h)/h\right)-\cos^{2}\left(g_{h}(\lambda)\right)\right)^{\frac{3}{2}};\]
since the\foreignlanguage{english}{ function $\lambda\mapsto\exp\left(2\pi\varepsilon(\lambda h)/h\right)>0$
is equal to $O(1)$ on the compact set }$\left[-1,1\right],$ \foreignlanguage{english}{we
see easily that on the compact set }$\left[-1,1\right]$ we have\foreignlanguage{english}{
}\[
\lambda\mapsto\frac{1}{D_{h}(\lambda)}=O(1).\]
Next, on the set $\left[-1,1\right]$\foreignlanguage{english}{\[
\lambda\mapsto\frac{\partial^{2}}{\partial\lambda^{2}}\left[\arccos\left(\frac{\cos\left(g_{h}(\lambda)\right)}{\sqrt{1+\exp\left(2\pi\varepsilon(\lambda h)/h\right)}}\right)\right]=O(1).\]
}Then on the compact set $\left[-1,1\right]$ we obtain that $\lambda\mapsto\mathcal{Y}_{h}^{\prime\prime}(\lambda)=O(1).$
For finish, since, uniformly on the compact set $[-1,1]$ we have
the equality\[
\mathcal{Y}_{h}^{\prime}(\lambda)=\frac{\ln(h)}{\sqrt{-V^{''}(0)}}+O(1);\]
we deduce that $\lambda\mapsto\left(\mathcal{A}_{h}^{\prime\prime}\circ\mathcal{Y}_{h}\right)(\lambda)=O\left|\frac{1}{\ln(h)^{3}}\right|$
on $[-1,1].$ 
\end{proof}

\subsection{Definition of a first order approximation time scale}

We have the following approximation result :
\begin{prop}
Let $\alpha$ a real number such that $\alpha<3-2\gamma$. Then, uniformly
for all $t\in\left[0,\left|\ln(h)\right|^{\alpha}\right]$ we have
:\[
\mathbf{a}(t)={\displaystyle \sum_{n\in\mathbb{N}}\left|a_{n}\right|^{2}e^{-it\left(\mathcal{A}_{h}(2\pi n_{0})+\mathcal{A}_{h}^{\prime}(2\pi n_{0})2\pi(n-n_{0})\right)}}+O\left(\left|\ln(h)\right|^{\alpha+2\gamma-3}\right).\]
\end{prop}
\begin{proof}
Let us introduce the difference $\varepsilon(t):=\varepsilon(t,h)$
defined by : \[
\varepsilon(t):=\left|{\displaystyle \sum_{n\in\mathbb{N}}\left|a_{n}\right|^{2}e^{-it\mathcal{A}_{h}(2\pi n)}}-{\displaystyle \sum_{n\in\mathbb{N}}\left|a_{n}\right|^{2}e^{-it\left(\mathcal{A}_{h}(2\pi n_{0})+\mathcal{A}_{h}^{\prime}(2\pi n_{0})2\pi(n-n_{0})\right)}}\right|.\]
Taylor-Lagrange's formula give the existence of a real number $\zeta=\zeta(n,h,E)\in\mathcal{Y}_{h}\left([-1,1]\right)$
such that \[
\mathcal{A}_{h}(2\pi n)=\mathcal{A}_{h}(2\pi n_{0})+\mathcal{A}_{h}^{\prime}(2\pi n_{0})2\pi(n-n_{0})+\mathcal{A}_{h}^{\prime\prime}(\zeta)2\pi^{2}(n-n_{0})^{2};\]
hence for all $t\geq0$ we get \[
\varepsilon(t)=\left|{\displaystyle \sum_{n\in\mathbb{N}}\left|a_{n}\right|^{2}e^{-it\left(\mathcal{A}_{h}(2\pi n_{0})+\mathcal{A}_{h}^{\prime}(2\pi n_{0})2\pi(n-n_{0})+\mathcal{A}_{h}^{\prime\prime}(\zeta)2\pi^{2}(n-n_{0})^{2}\right)}}\right.\]
\[
\left.-{\displaystyle \sum_{n\in\mathbb{N}}\left|a_{n}\right|^{2}e^{-it\left(\mathcal{A}_{h}(2\pi n_{0})+\mathcal{A}_{h}^{\prime}(2\pi n_{0})2\pi(n-n_{0})\right)}}\right|\]
\[
=\left|\sum_{n\in\mathbb{N}}\left|a_{n}\right|^{2}e^{-it\left(\mathcal{A}_{h}(2\pi n_{0})+\mathcal{A}_{h}^{\prime}(2\pi n_{0})2\pi(n-n_{0})\right)}\left[e^{-it\mathcal{A}_{h}^{\prime\prime}(\zeta)2\pi^{2}(n-n_{0})^{2}}-1\right]\right|.\]
With the sets $\Gamma,$$\Delta$ and by triangular inequality, we
obtain for all $t\geq0$ \[
\varepsilon(t)\leq\left|{\displaystyle \sum_{n\in\Delta}\left|a_{n}\right|^{2}e^{-it\left(\mathcal{A}_{h}(2\pi n_{0})+\mathcal{A}_{h}^{\prime}(2\pi n_{0})2\pi(n-n_{0})\right)}\left[e^{-it\mathcal{A}_{h}^{\prime\prime}(\zeta)2\pi^{2}(n-n_{0})^{2}}-1\right]}\right|\]
\begin{equation}
+\left|{\displaystyle \sum_{n\in\Gamma}\left|a_{n}\right|^{2}e^{-it\left(\mathcal{A}_{h}(2\pi n_{0})+\mathcal{A}_{h}^{\prime}(2\pi n_{0})2\pi(n-n_{0})\right)}\left[e^{-it\mathcal{A}_{h}^{\prime\prime}(\zeta)2\pi^{2}(n-n_{0})^{2}}-1\right]}\right|.\label{eq:-1}\end{equation}
First, look at the term the right part of the inequality above, for
all $t\geq0$ we have \foreignlanguage{english}{\[
\left|{\displaystyle \sum_{n\in\Gamma}\left|a_{n}\right|^{2}e^{-it\left(\mathcal{A}_{h}(2\pi n_{0})+\mathcal{A}_{h}^{\prime}(2\pi n_{0})2\pi(n-n_{0})\right)}\left[e^{-it\mathcal{A}_{h}^{\prime\prime}(\zeta)2\pi^{2}(n-n_{0})^{2}}-1\right]}\right|\]
}\[
\leq2\sum_{n\in\Gamma}\left|a_{n}\right|^{2}=O\left|\frac{1}{\ln(h)^{\infty}}\right|\]
by the lemma 4.15. 

Now look at the left term of the inequality above. With the lemma
5.1: for all integer $n\in\Delta$ and for all real number $t\in\left[0,\left|\ln(h)\right|^{\alpha}\right]$
we have\[
t\mathcal{A}_{h}^{\prime\prime}(\zeta)2\pi^{2}(n-n_{0})^{2}\leq M\left|\ln(h)\right|^{\alpha+2\gamma-3}\]
where $M>0$. Consequently, for all integer $n\in\Delta$ and for
all $t\in\left[0,\left|\ln(h)\right|^{\alpha}\right]$ we have\[
e^{-it\mathcal{A}_{h}^{\prime\prime}(\zeta)2\pi^{2}(n-n_{0})^{2}}-1=O\left(\left|\ln(h)\right|^{\alpha+2\gamma-3}\right);\]
hence for all $t\in\left[0,\left|\ln(h)\right|^{\alpha}\right]$\foreignlanguage{english}{\[
\left|{\displaystyle \sum_{n\in\Delta}\left|a_{n}\right|^{2}e^{-it\left(\mathcal{A}_{h}(2\pi n_{0})+\mathcal{A}_{h}^{\prime}(2\pi n_{0})2\pi(n-n_{0})\right)}\left[e^{-it\mathcal{A}_{h}^{\prime\prime}(\zeta)2\pi^{2}(n-n_{0})^{2}}-1\right]}\right|\]
}\[
\leq O\left(\left|\ln(h)\right|^{\alpha+2\gamma-3}\right)\sum_{n\in\Delta}\left|a_{n}\right|^{2}\]
\[
\leq O\left(\left|\ln(h)\right|^{\alpha+2\gamma-3}\right)\sum_{n\in\mathbb{N}}\left|a_{n}\right|^{2}=O\left(\left|\ln(h)\right|^{\alpha+2\gamma-3}\right).\]
So, for all $t\in\left[0,\left|\ln(h)\right|^{\alpha}\right]$ we
get finally \[
\varepsilon(t)=O\left(\left|\ln(h)\right|^{\alpha+2\gamma-3}\right).\]

\end{proof}
We conclude that, the principal term of the partial autocorrelation
$\mathbf{a}(t)$ in the time-scale $\left[0,\left|\ln(h)\right|^{\alpha}\right]$
is:
\begin{defn}
The principal term of the partial autocorrelation $\mathbf{a}(t)$
is defined by :\[
\mathbf{a_{1}\,:\,}t\mapsto\sum_{n\in\mathbb{N}}\left|a_{n}\right|^{2}e^{-it\left(\mathcal{A}_{h}(2\pi n_{0})+\mathcal{A}_{h}^{\prime}(2\pi n_{0})2\pi(n-n_{0})\right)}.\]
We also define the function $\mathbf{\widetilde{a_{1}}}$ by \[
\mathbf{\mathbf{\widetilde{a_{1}}}\,:\,}t\mapsto\sum_{n\in\mathbb{N}}\left|a_{n}\right|^{2}e^{-it\mathcal{A}_{h}^{\prime}(2\pi n_{0})2\pi(n-n_{0})}.\]

\end{defn}
Hence\[
\mathbf{\mathbf{a_{1}}}(t)=e^{-it\mathcal{A}_{h}(2\pi n_{0})}\mathbf{\mathbf{\widetilde{a_{1}}}}(t)\mathbf{\mathbf{.}}\]
Now, we study this serie in details.

\subsection{Periodicity of the principal term}

Clearly, the function $t\mapsto\left|\mathbf{\mathbf{a_{1}}}(t)\right|$
is $1/\mathcal{A}_{h}^{\prime}(2\pi n_{0})$-periodic, so the approximation
$\mathbf{a}_{\mathbf{1}}$ defines an important characteristic time
scale :
\begin{defn}
Let us define the hyperbolic period $T_{hyp}=T_{hyp}(h,E)$ by \[
T_{hyp}:=\frac{1}{\mathcal{A}_{h}^{\prime}(2\pi n_{0})}.\]

\end{defn}
Note that, since\foreignlanguage{english}{$\mathcal{A}_{h}^{\prime}(\zeta)=2\pi/\left(\frac{\left|\ln(h)\right|}{\sqrt{-V^{''}(0)}}+O(1)\right);$}
hence for $h\rightarrow0$ we have $T_{hyp}\sim K\left|\ln(h)\right|;$
where $K>0$ (not depends on $h$).

\subsection{Geometric interpretation of the period}

The period $T_{hyp}$ correspond to the period of the classical flow.
The term $\left|\ln(h)\right|$ is the signature of the hyperbolic
singularity : indeed, the term $\left|\ln(h)\right|$ correspond to
the period of the classical flow with a initial point on the disc
$D(O,\sqrt{h}):=\left\{ (x,\xi)\in\mathbb{R}^{2},\,\sqrt{x^{2}+\xi^{2}}=\sqrt{h}\right\} $.
\begin{thm}
\textbf{{[}Lab3{]}. }Let us consider the point $m_{h}=\left(\sqrt{h},0\right)\in T^{*}\mathbb{R}$.
Then the Hamiltonian's flow associated to the function $p$ and with
initial point $m_{h}$ is periodic and the period $\tau_{h}$ verify,
for $h\rightarrow0$, the following equivalent :\[
\tau_{h}\sim\frac{\left|\ln\left(h\right)\right|}{K}\]
where the constant $K\neq0$ and not depend on $h$.\end{thm}
\begin{rem}
See also the paper \textbf{{[}DB-Ro{]} }for a similar result.
\end{rem}

\subsection{Comparison between hyperbolic period and the time scale $\left[0,\left|\ln(h)\right|^{\alpha}\right]$}

Since $\gamma<1$, we have $1<3-2\gamma$; hence there exist a real
number $\alpha$ such that $\alpha\in\left]1,3-2\gamma\right[$. Consequently,
we get \[
\left|\ln(h)\right|<\left|\ln(h)\right|^{\alpha}<\left|\ln(h)\right|^{3-2\gamma}.\]
So, we can make a {}``good choice'' for $\alpha$ : for $h$ small
enough we have : \[
\left[0,T_{hyp}\right]\subset\left[0,\left|\ln(h)\right|^{\alpha}\right].\]

\subsection{Behaviour of autocorrelation function on a hyperbolic period }

Now, let us study in details the function $\mathbf{a_{1}}(t)$ on
the period \textit{$\left[0,T_{hyp}\right]$. }Start by a technical
proposition :
\begin{prop}
For all $t\geq0$, we have the equality :\[
\sum_{n\in\mathbb{\mathbb{Z}}}\left|a_{n}\right|^{2}e^{-it2\pi(n-n_{0})\frac{1}{T_{hyp}}}=\frac{1}{\mathfrak{F}\left(\chi^{2}\right)(0)}\sum_{\ell\in\mathbb{Z}}\mathfrak{F}\left(\chi^{2}\right)\left(-\left|\ln h\right|^{1-\gamma^{\prime}}\left(\ell+\frac{t}{T_{hyp}}\right)\right).\]
\end{prop}
\begin{proof}
Let us consider the function $\Omega_{t}$ defined by \[
\Omega_{t}:\left\{ \begin{array}{cc}
\mathbb{R}\rightarrow\mathbb{\mathbb{C}}\\
\\x\mapsto\left|a_{x}\right|^{2}e^{-it2\pi(x-n_{0})\frac{1}{T_{hyp}}}\end{array}\right.\]
where $t\in\mathbb{R}$ is a parameter, and let us recall $a_{x}$
is defined by\[
a_{x}=\frac{1}{\sqrt{\mathfrak{F}\left(\chi^{2}\right)(0)}\left|\ln h\right|^{\frac{1-\gamma^{\prime}}{2}}}\chi\left(\frac{x-n_{0}}{\left|\ln h\right|^{1-\gamma^{\prime}}}\right);\]
thus \[
\sum_{n\in\mathbb{\mathbb{Z}}}\left|a_{n}\right|^{2}e^{-it2\pi(n-n_{0})\frac{1}{T_{hyp}}}=\sum_{n\in\mathbb{Z}}\Omega_{t}(n).\]
So clearly, the function $\Omega_{t}\in\mathcal{S}(\mathbb{R})$,
then the Fourier transform $\mathfrak{F}\left(\Omega_{t}\right)$
is equal, for all $\zeta\in\mathbb{R}$, to \[
\mathfrak{F}\left(\Omega_{t}\right)(\zeta)=\frac{1}{\mathfrak{F}\left(\chi^{2}\right)(0)\left|\ln h\right|^{1-\gamma^{\prime}}}e^{it2\pi n_{0}\frac{1}{T_{hyp}}}\mathfrak{F}\left(\chi^{2}\left(\frac{x-n_{0}}{\left|\ln h\right|^{1-\gamma^{\prime}}}\right)e^{-it2\pi x\frac{1}{T_{hyp}}}\right)(\zeta).\]
\[
=\frac{1}{\mathfrak{F}\left(\chi^{2}\right)(0)}e^{-2i\pi n_{0}\zeta}\mathfrak{F}\left(\chi^{2}\right)\left(-\left|\ln h\right|^{1-\gamma^{\prime}}\left(\zeta+\frac{t}{T_{hyp}}\right)\right).\]
With the Poisson formula we get\[
\sum_{n\in\mathbb{Z}}\Omega_{t}(n)=\sum_{\ell\in\mathbb{Z}}\mathfrak{F}\left(\Omega_{t}\right)(\ell)\]
\textit{\[
=\frac{1}{\mathfrak{F}\left(\chi^{2}\right)(0)}\sum_{\ell\in\mathbb{Z}}\mathfrak{F}\left(\chi^{2}\right)\left(-\left|\ln h\right|^{1-\gamma^{\prime}}\left(\ell+\frac{t}{T_{hyp}}\right)\right).\]
}
\end{proof}
So, now, our goal is to study the behaviour of the serie :\[
t\mapsto\frac{1}{\mathfrak{F}\left(\chi^{2}\right)(0)}\sum_{\ell\in\mathbb{Z}}\mathfrak{F}\left(\chi^{2}\right)\left(-\left|\ln h\right|^{1-\gamma^{\prime}}\left(\ell+\frac{t}{T_{hyp}}\right)\right).\]
Since the function $\mathfrak{F}\left(\chi^{2}\right)\in\mathcal{S}(\mathbb{R})$,
it's clear that the only index $\ell\in\mathbb{Z}$ such that $\left(\ell+\frac{t}{T_{hyp}}\right)$
is close to zero are important. More precisely : 
\begin{defn}
For all $t\in\mathbb{R}$, let us define $\ell(t)=\ell(t,h,E)$ the
closest integer to the real number $-t/T_{hyp}$; e.g :\[
\ell(t)+\frac{t}{T_{hyp}}=d\left(t,T_{hyp}\mathbb{Z}\right);\]
where $d(.,.)$ denote the Euclidiean distance on $\mathbb{R}$.\end{defn}
\begin{rem}
Without loss of generality, we may suppose $\ell(t)$ is unique. On
the other hand, for all integer $\ell\in\mathbb{Z}$ such that $\ell\neq\ell(t)$
we get :\[
\left|\ell+\frac{t}{T_{hyp}}\right|\geq\frac{1}{2}.\]
\end{rem}
\begin{lem}
For a function $\varphi\in\mathcal{S}(\mathbb{R})$ and $\varepsilon\in\left]0,1\right]$
then, uniformly for $u\in\mathbb{R}$ we have : \[
\sum_{\ell\in\mathbb{\mathbb{Z}},\,\left|\ell+u\right|\geq\frac{1}{2}}\left|\varphi\left(\frac{\ell+u}{\varepsilon}\right)\right|=O(\varepsilon^{\infty}).\]
\end{lem}
\begin{proof}
We see easily that, uniformly for $u\in\mathbb{R}$ we have \[
\sum_{\ell\in\mathbb{\mathbb{Z}},\,\left|\ell+u\right|\geq\frac{1}{2}}\left|\varphi\left(\frac{\ell+u}{\varepsilon}\right)\right|=O(1).\]
Indeed, since $\varepsilon\in\left]0,1\right]$, without loss of generality,
we may suppose $u\in\left[-1,1\right]$. And since for all $u\in\left[-1,1\right]$
we have\[
\left|\varphi\left(\frac{\ell+u}{\varepsilon}\right)\right|\leq\frac{M}{1+\frac{(\ell+u)^{2}}{\varepsilon^{2}}}=\frac{\varepsilon^{2}M}{\varepsilon^{2}+(\ell+u)^{2}}\]
\[
\leq\frac{M}{(\ell+u)^{2}}=\frac{M}{\left|\ell^{2}-2\left|\ell\right|\right|};\]
and we conclude.

Next,\[
\sum_{\ell\in\mathbb{\mathbb{Z}},\,\left|\ell+u\right|\geq\frac{1}{2}}\left|\varphi\left(\frac{\ell+u}{\varepsilon}\right)\right|=\sum_{\ell\in\mathbb{\mathbb{Z}},\,\left|\ell+u\right|\geq\frac{1}{2}}\left(\frac{\ell+u}{\varepsilon}\right)^{2N}\left|\varphi\left(\frac{\ell+u}{\varepsilon}\right)\right|\frac{\varepsilon^{2N}}{(\ell+u)^{2N}}\]
\[
\leq\varepsilon^{2N}4^{N}\sum_{\ell\in\mathbb{\mathbb{Z}}}\left(\frac{\ell+u}{\varepsilon}\right)^{2N}\left|\varphi\left(\frac{\ell+u}{\varepsilon}\right)\right|.\]
To conclude the proof, we apply that to the function $\psi(x):=x^{2N}\varphi(x)$.\end{proof}
\begin{thm}
Uniformly for $t\geq0$ we have :\[
\sum_{n\in\mathbb{N}}\left|a_{n}\right|^{2}e^{\frac{-it2\pi(n-n_{0})}{T_{hyp}}}=\frac{1}{\mathfrak{F}\left(\chi^{2}\right)(0)}\mathfrak{F}\left(\chi^{2}\right)\left(-\left|\ln h\right|^{1-\gamma^{\prime}}d\left(t,T_{hyp}\mathbb{Z}\right)\right)+O\left|\frac{1}{\ln(h)^{\infty}}\right|.\]
\end{thm}
\begin{proof}
Since the proposition 5.7 and with the lemma above we get for all
\foreignlanguage{english}{$t\geq0$}\textit{\[
\sum_{n\in\mathbb{\mathbb{Z}}}\left|a_{n}\right|^{2}e^{-it2\pi(n-n_{0})\frac{t}{T_{hyp}}}=\frac{1}{\mathfrak{F}\left(\chi^{2}\right)(0)}\sum_{\ell\in\mathbb{Z}}\mathfrak{F}\left(\chi^{2}\right)\left(-\left|\ln h\right|^{1-\gamma^{\prime}}\left(\ell+\frac{t}{T_{hyp}}\right)\right)\]
\[
=\frac{1}{\mathfrak{F}\left(\chi^{2}\right)(0)}\mathfrak{F}\left(\chi^{2}\right)\left(-\left|\ln h\right|^{1-\gamma^{\prime}}d\left(t,T_{hyp}\mathbb{Z}\right)\right)+O\left|\frac{1}{\ln(h)^{\infty}}\right|.\]
}Next, for all $t\geq0$ \textit{\[
\sum_{n\in\mathbb{N}}\left|a_{n}\right|^{2}e^{-it2\pi(n-n_{0})\frac{1}{T_{hyp}}}\]
\[
=\sum_{n\in\mathbb{\mathbb{Z}}}\left|a_{n}\right|^{2}e^{-it2\pi(n-n_{0})\frac{1}{T_{hyp}}}-\sum_{n=-\infty}^{-1}\left|a_{n}\right|^{2}e^{-it2\pi(n-n_{0})\frac{1}{T_{hyp}}};\]
}

\textit{\[
=\frac{1}{\mathfrak{F}\left(\chi^{2}\right)(0)}\sum_{\ell\in\mathbb{Z}}\mathfrak{F}\left(\chi^{2}\right)\left(-\left|\ln h\right|^{1-\gamma^{\prime}}\left(\ell+\frac{t}{T_{hyp}}\right)\right)\]
}\[
-\sum_{n=-\infty}^{-1}\left|a_{n}\right|^{2}e^{-it2\pi(n-n_{0})\frac{1}{T_{hyp}}}.\]
Hence, by triangular inequality, we have for all $t\geq0$ \textit{\[
\left|\sum_{n\in\mathbb{N}}\left|a_{n}\right|^{2}e^{-it2\pi(n-n_{0})\frac{1}{T_{hyp}}}-\frac{1}{\mathfrak{F}\left(\chi^{2}\right)(0)}\mathfrak{F}\left(\chi^{2}\right)\left(-\left|\ln h\right|^{1-\gamma^{\prime}}d\left(t,T_{hyp}\mathbb{Z}\right)\right)\right|\]
\[
\leq\sum_{\ell\neq\ell(t)}\left|\frac{1}{\mathfrak{F}\left(\chi^{2}\right)(0)}\mathfrak{F}\left(\chi^{2}\right)\left(-\left|\ln h\right|^{1-\gamma^{\prime}}\left(\ell+\frac{t}{T_{hyp}}\right)\right)\right|+\sum_{n=1}^{+\infty}\left|a_{-n}\right|^{2}.\]
}Since, for all $t\geq0$, we have \[
\frac{1}{\mathfrak{F}\left(\chi^{2}\right)(0)}\sum_{\ell\neq\ell(t)}\mathfrak{F}\left(\chi^{2}\right)\left(-\left|\ln h\right|^{1-\gamma^{\prime}}\left(\ell+\frac{t}{T_{hyp}}\right)\right)=O\left|\frac{1}{\ln(h)^{\infty}}\right|;\]
and\textit{\[
\sum_{n=1}^{+\infty}\left|a_{-n}\right|^{2}=O\left|\frac{1}{\ln(h)^{\infty}}\right|.\]
}So we prove the theorem.
\end{proof}
Now, we can formulate the main result of the section 5 :
\begin{cor}
We have :

\textbf{(i)} for $t$ such that $t\in T_{hyp}\mathbb{Z}$ we get\[
\sum_{n\in\mathbb{N}}\left|a_{n}\right|^{2}e^{-it\frac{2\pi(n-n_{0})}{T_{hyp}}}=1.\]
\textbf{(ii)} For all $\varepsilon>0$ such that $\varepsilon<1-\gamma^{'}$,
and for $t$ such that $\left|d\left(t,T_{hyp}\mathbb{Z}\right)\right|>\left|\ln h\right|^{\gamma^{\prime}-1+\varepsilon}$;
we get \[
\sum_{n\in\mathbb{N}}\left|a_{n}\right|^{2}e^{-it\frac{2\pi(n-n_{0})}{T_{hyp}}}=O\left|\frac{1}{\ln(h)^{\infty}}\right|.\]
\end{cor}
\begin{proof}
The first point is clear. For the second : if\textit{ \[
\left|-\left|\ln h\right|^{1-\gamma^{\prime}}\left(\ell(t)+t\mathcal{A}_{h}^{\prime}(2\pi n_{0})\right)\right|>\left|\ln h\right|^{\varepsilon}\]
} thus, we have :\[
\left|\mathfrak{F}\left(\chi^{2}\right)\left(-\left|\ln h\right|^{1-\gamma^{\prime}}\left(\ell(t)+t\mathcal{A}_{h}^{\prime}(2\pi n_{0})\right)\right)\right|\leq\frac{B_{k}}{\left(1+\left|\ln h\right|^{\varepsilon}\right)^{k}}\]
\[
\leq B_{k}\left|\ln h\right|^{-\varepsilon k}.\]
for all integer $k\geq1;$ so we prove the second point.
\end{proof}

\section{Second order approximation : revival period}

\subsection{Introduction}

In this this section, we use a second order approximation of the function
$t\mapsto\mathbf{a}(t)$, indeed by a Taylor's formula we have for
all $t\geq0$\[
\mathbf{a}(t)={\displaystyle \sum_{n\in\mathbb{N}}\left|a_{n}\right|^{2}e^{-it\left(\mathcal{A}_{h}(2\pi n_{0})+\mathcal{A}_{h}^{\prime}(2\pi n_{0})2\pi(n-n_{0})+\mathcal{A}_{h}^{\prime\prime}(2\pi n_{0})2\pi^{2}(n-n_{0})^{2}+\mathcal{A}_{h}^{(3)}(\zeta)\frac{(2\pi)^{3}}{6}(n-n_{0})^{3}\right)}}\]
where $\zeta=\zeta(n,h,E)$.

We need the :
\begin{lem}
Uniformly, on the compact set $[-1,1]$ we have\[
\lambda\mapsto\left(\mathcal{A}_{h}^{(3)}\circ\mathcal{Y}_{h}\right)(\lambda)=O\left|\frac{1}{\ln(h)^{4}}\right|.\]
\end{lem}
\begin{proof}
With the derivatives formulas, we have for all $x\in\mathcal{Y}_{h}\left([-1,1]\right)$\[
\mathcal{A}_{h}^{(3)}(x)=\frac{-\left(\mathcal{Y}_{h}^{(3)}\circ\mathcal{A}_{h}\right)(x)}{\left(\mathcal{Y}_{h}^{\prime}\circ\mathcal{A}_{h}\right)^{4}(x)}+\frac{3\left(\mathcal{Y}_{h}^{\prime\prime}\circ\mathcal{A}_{h}\right)^{2}(x)}{\left(\mathcal{Y}_{h}^{\prime}\circ\mathcal{A}_{h}\right)^{5}(x)};\]
hence, for all $\lambda\in[-1,1]$ we get \[
\left(\mathcal{A}_{h}^{(3)}\circ\mathcal{Y}_{h}\right)(\lambda)=\frac{-\mathcal{Y}_{h}^{(3)}(\lambda)}{\left(\mathcal{Y}_{h}^{\prime}\right)^{4}(\lambda)}+\frac{3\left(\mathcal{Y}_{h}^{\prime\prime}\right)^{2}(\lambda)}{\left(\mathcal{Y}_{h}^{\prime}\right)^{5}(\lambda)}.\]
First, let us estimate the function $\lambda\mapsto\mathcal{Y}_{h}^{(3)}(\lambda)$.
For all $\lambda\in\left[-1,1\right]$ we have\[
\mathcal{Y}_{h}^{(3)}(\lambda)=-h^{3}\frac{\theta_{+}^{(3)}(\lambda h)+\theta_{-}^{(3)}(\lambda h)}{2}+h^{2}\varepsilon^{(3)}(\lambda h)\ln(h)\]
\[
+\frac{\partial^{3}}{\partial\lambda^{3}}\left[\arg\left(\Gamma\left(\frac{1}{2}+i\frac{\varepsilon(\lambda h)}{h}\right)\right)\right]-\frac{\partial^{3}}{\partial\lambda^{3}}\left[\arccos\left(\frac{\cos\left(g_{h}(\lambda)\right)}{\sqrt{1+\exp\left(2\pi\varepsilon(\lambda h)/h\right)}}\right)\right].\]
Uniformly, on the compact set $\left[-1,1\right]$ we have \[
\lambda\mapsto-h^{3}\frac{\theta_{+}^{(3)}(\lambda h)+\theta_{-}^{(3)}(\lambda h)}{2}=O(h^{2})\]
and \[
\lambda\mapsto h^{2}\varepsilon^{(3)}(\lambda h)\ln(h)=O(h^{2}\ln(h)).\]
Next, let us estimate the function $\lambda\mapsto\frac{\partial^{3}}{\partial\lambda^{3}}\left[\arg\left(\Gamma\left(\frac{1}{2}+i\frac{\varepsilon(\lambda h)}{h}\right)\right)\right]$
: for all $\lambda\in[-1,1]$ \[
\frac{\partial^{3}}{\partial\lambda^{3}}\left[\arg\left(\Gamma\left(\frac{1}{2}+i\frac{\varepsilon(\lambda h)}{h}\right)\right)\right]\]
\[
=\frac{\partial}{\partial\lambda}\left[h\varepsilon^{\prime\prime}(\lambda h)\textrm{Re}\left(\Psi\left(\frac{1}{2}+i\frac{\varepsilon(\lambda h)}{h}\right)\right)-\left(\varepsilon^{\prime}(\lambda h)\right)^{2}\textrm{Im}\left(\Psi^{(1)}\left(\frac{1}{2}+i\frac{\varepsilon(\lambda h)}{h}\right)\right)\right]\]
\[
=h^{2}\varepsilon^{(3)}(\lambda h)\textrm{Re}\left(\Psi\left(\frac{1}{2}+i\frac{\varepsilon(\lambda h)}{h}\right)\right)+h\varepsilon^{\prime\prime}(\lambda h)\frac{\partial}{\partial\lambda}\left[\textrm{Re}\left(\Psi\left(\frac{1}{2}+i\frac{\varepsilon(\lambda h)}{h}\right)\right)\right]\]
\[
-2h\varepsilon^{\prime}(\lambda h)\varepsilon^{\prime\prime}(\lambda h)\textrm{Im}\left(\Psi^{(1)}\left(\frac{1}{2}+i\frac{\varepsilon(\lambda h)}{h}\right)\right)-\left(\varepsilon^{\prime}(\lambda h)\right)^{2}\frac{\partial}{\partial\lambda}\left[\textrm{Im}\left(\Psi^{(1)}\left(\frac{1}{2}+i\frac{\varepsilon(\lambda h)}{h}\right)\right)\right]\]
\[
=h^{2}\varepsilon^{(3)}(\lambda h)\textrm{Re}\left(\Psi\left(\frac{1}{2}+i\frac{\varepsilon(\lambda h)}{h}\right)\right)-3h\varepsilon^{\prime\prime}(\lambda h)\varepsilon^{\prime}(\lambda h)\textrm{Im}\left(\Psi^{(1)}\left(\frac{1}{2}+i\frac{\varepsilon(\lambda h)}{h}\right)\right)\]
\[
+\left(\varepsilon^{\prime}(\lambda h)\right)^{3}\textrm{Re}\left(\Psi^{(2)}\left(\frac{1}{2}+i\frac{\varepsilon(\lambda h)}{h}\right)\right);\]
thus, the function $\lambda\mapsto\frac{\partial^{3}}{\partial\lambda^{3}}\left[\arg\left(\Gamma\left(\frac{1}{2}+i\frac{\varepsilon(\lambda h)}{h}\right)\right)\right]$
is equal to a $O(1)$ on the compact set \foreignlanguage{english}{{[}-1,1{]}.}
Now, for finish, let us estimate the function\[
\lambda\mapsto\frac{\partial^{3}}{\partial\lambda^{3}}\left[\arccos\left(\frac{\cos\left(g_{h}(\lambda)\right)}{\sqrt{1+\exp\left(2\pi\varepsilon(\lambda h)/h\right)}}\right)\right];\]
with the notations introduced in the proof of lemma 5.1, for all $\lambda\in\left[-1,1\right]$
we have \[
\frac{\partial^{3}}{\partial\lambda^{3}}\left[\arccos\left(\frac{\cos\left(g_{h}(\lambda)\right)}{\sqrt{1+\exp\left(2\pi\varepsilon(\lambda h)/h\right)}}\right)\right]=\frac{\partial}{\partial\lambda}\left(\frac{N_{h}}{D_{h}}\right)\]
\[
=\frac{N_{h}^{\prime}(\lambda)D_{h}(\lambda)-N_{h}(\lambda)D_{h}^{\prime}(\lambda)}{D_{h}^{2}(\lambda)};\]
we have seen into the proof of lemma 5.1, that, on the compact set
$\left[-1,1\right]$ \[
\lambda\mapsto\frac{1}{D_{h}^{2}(\lambda)}=O(1).\]
By the same argument as into the prof of lemma 5.1, we show that $\lambda\mapsto N_{h}^{\prime}(\lambda)D_{h}(\lambda)-N_{h}(\lambda)D_{h}^{\prime}(\lambda)$
is to equal to a $O(1)$ on the compact set $\left[-1,1\right]$.
Consequently for all $\lambda\in\left[-1,1\right]$ the function \[
\lambda\mapsto\frac{\partial^{3}}{\partial\lambda^{3}}\left[\arccos\left(\frac{\cos\left(g_{h}(\lambda)\right)}{\sqrt{1+\exp\left(2\pi\varepsilon(\lambda h)/h\right)}}\right)\right]\]
is to equal to a $O(1)$ on the compact set $\left[-1,1\right]$.
Thus, on $\left[-1,1\right]$ we have $\lambda\mapsto\mathcal{Y}_{h}^{(3)}(\lambda)=O(1).$ 

Since, we have, uniformly, on the compact set $[-1,1]$ the equality
\[
\mathcal{Y}_{h}^{\prime}(\lambda)=\frac{\ln(h)}{\sqrt{-V^{''}(0)}}+O(1),\]
we deduce, that the function \[
\lambda\mapsto\frac{-\mathcal{Y}_{h}^{(3)}(\lambda)}{\left(\mathcal{Y}_{h}^{\prime}\right)^{4}(\lambda)}=O\left|\frac{1}{\ln(h)^{4}}\right|\]
on the compact set $\left[-1,1\right]$. 

In the proof of lemma 5.1 we have seen, that, on the compact set $[-1,1]$
:\[
\lambda\mapsto\mathcal{Y}_{h}^{\prime\prime}(\lambda)=O\left|\frac{1}{\ln(h)^{3}}\right|;\]
hence, on the compact set $[-1,1]$, we get 

\[
\lambda\mapsto\frac{\left(\mathcal{Y}_{h}^{\prime\prime}\right)^{2}(\lambda)}{\left(\mathcal{Y}_{h}^{\prime}\right)^{5}(\lambda)}=O\left|\frac{1}{\ln(h)^{5}}\right|;\]
thus \[
\lambda\mapsto\left(\mathcal{A}_{h}^{(3)}\circ\mathcal{Y}_{h}\right)(\lambda)=\frac{-\mathcal{Y}_{h}^{(3)}(\lambda)}{\left(\mathcal{Y}_{h}^{\prime}\right)^{4}(\lambda)}+\frac{3\left(\mathcal{Y}_{h}^{\prime\prime}\right)^{2}(\lambda)}{\left(\mathcal{Y}_{h}^{\prime}\right)^{5}(\lambda)}\]
is equal to a $O\left|\frac{1}{\ln(h)^{4}}\right|$ on the set $[-1,1]$
.\end{proof}
\begin{notation}
Let us denote by $Q_{2}(X)=Q_{2}(h,n_{0},X)$ the polynomial of $\mathbb{R}_{2}[X]$
defined by :\[
Q_{2}(X):=\mathcal{A}_{h}(2\pi n_{0})+\mathcal{A}_{h}^{\prime}(2\pi n_{0})2\pi(X-n_{0})+\mathcal{A}_{h}^{\prime\prime}(2\pi n_{0})2\pi^{2}(X-n_{0})^{2}.\]

\end{notation}

\subsection{Definition of a new time scale}
\begin{prop}
Let $\beta$ be a real number such that $\beta<4-3\gamma$. Then,
uniformly for all $t\in\left[0,\left|\ln(h)\right|^{\beta}\right]$
we have : \[
\mathbf{a}(t)={\displaystyle \sum_{n\in\mathbb{N}}\left|a_{n}\right|^{2}e^{-itQ_{2}(n)}}+O\left(\left|\ln(h)\right|^{\beta+3\gamma-4}\right).\]
\end{prop}
\begin{proof}
Let us introduce the difference $\varepsilon(t):=\varepsilon(t,h)$
defined by :\[
\varepsilon(t):=\left|{\displaystyle \mathbf{a}(t)}-{\displaystyle \sum_{n\in\mathbb{N}}\left|a_{n}\right|^{2}e^{-itQ_{2}(n)}}\right|.\]
Taylor-Lagrange's formula gives the existence of a real number $\zeta=\zeta(n,h,E)\in\mathcal{Y}_{h}\left([-1,1]\right)$
such that :\[
\mathcal{A}_{h}(2\pi n)=\mathcal{A}_{h}(2\pi n_{0})+\mathcal{A}_{h}^{\prime}(2\pi n_{0})2\pi(n-n_{0})\]
\[
+\mathcal{A}_{h}^{\prime\prime}(2\pi n_{0})2\pi^{2}(n-n_{0})^{2}+\mathcal{A}_{h}^{(3)}(\zeta)\frac{(2\pi)^{3}}{6}(n-n_{0})^{3};\]
hence, for all $t\geq0$ we get\[
\varepsilon(t)=\left|\sum_{n\in\mathbb{N}}\left|a_{n}\right|^{2}e^{-it\left(Q_{2}(n)+\mathcal{A}_{h}^{(3)}(\zeta)\frac{(2\pi)^{3}}{6}(n-n_{0})^{3}\right)}-{\displaystyle \sum_{n\in\mathbb{N}}\left|a_{n}\right|^{2}e^{-itQ_{2}(n)}}\right|\]
\[
=\left|\sum_{n\in\mathbb{N}}\left|a_{n}\right|^{2}e^{-itQ_{2}(n)}\left[e^{-it\mathcal{A}_{h}^{(3)}(\zeta)\frac{(2\pi)^{3}}{6}(n-n_{0})^{3}}-1\right]\right|.\]
With the sets $\Gamma$, $\Delta$ and by triangular inequality, we
obtain for all $t\geq0$\[
\varepsilon(t)\leq\left|{\displaystyle \sum_{n\in\Delta}\left|a_{n}\right|^{2}e^{-itQ_{2}(n)}\left[e^{-it\mathcal{A}_{h}^{(3)}(\zeta)\frac{(2\pi)^{3}}{6}(n-n_{0})^{3}}-1\right]}\right|\]
\[
+\left|{\displaystyle \sum_{n\in\Gamma}\left|a_{n}\right|^{2}e^{-itQ_{2}(n)}\left[e^{-it\mathcal{A}_{h}^{(3)}(\zeta)\frac{(2\pi)^{3}}{6}(n-n_{0})^{3}}-1\right]}\right|.\]
Since :\foreignlanguage{english}{\[
\left|{\displaystyle \sum_{n\in\Gamma}\left|a_{n}\right|^{2}e^{-itQ_{2}(n)}\left[e^{-it\mathcal{A}_{h}^{(3)}(\zeta)\frac{(2\pi)^{3}}{6}(n-n_{0})^{3}}-1\right]}\right|\leq2\sum_{n\in\Gamma}\left|a_{n}\right|^{2}=O\left|\frac{1}{\ln(h)^{\infty}}\right|\]
}by the lemma 4.15. 

On the other hand, with the lemma 6.1: for all integer $n\in\Delta$
and for all real number $t\in\left[0,\left|\ln(h)\right|^{\beta}\right]$\[
\left|t\mathcal{A}_{h}^{(3)}(\zeta)\frac{(2\pi)^{3}}{6}(n-n_{0})^{3}\right|\leq M\left|\ln(h)\right|^{\beta+3\gamma-4}\]
where $M>0$. Thus, for all integer $n\in\Delta$ and for all real
number $t\in\left[0,\left|\ln(h)\right|^{\beta}\right]$ we have \[
e^{-it\mathcal{A}_{h}^{(3)}(\zeta)\frac{(2\pi)^{3}}{6}(n-n_{0})^{3}}-1=O\left(\left|\ln(h)\right|^{\beta+3\gamma-4}\right)\]
hence for all $t\in\left[0,\left|\ln(h)\right|^{\beta}\right]$ we
get\foreignlanguage{english}{\[
\left|{\displaystyle \sum_{n\in\Delta}\left|a_{n}\right|^{2}e^{-itQ_{2}(n)}\left[e^{-it\mathcal{A}_{h}^{(3)}(\zeta)\frac{(2\pi)^{3}}{6}(n-n_{0})^{3}}-1\right]}\right|\]
}\[
\leq O\left(\left|\ln(h)\right|^{\beta+3\gamma-4}\right)\sum_{n\in\Delta}\left|a_{n}\right|^{2}\]
\[
\leq O\left(\left|\ln(h)\right|^{\beta+3\gamma-4}\right)\sum_{n\in\mathbb{N}}\left|a_{n}\right|^{2}=O\left(\left|\ln(h)\right|^{\beta+3\gamma-4}\right).\]
So, for all $t\in\left[0,\left|\ln(h)\right|^{\beta}\right]$ we get
finally \[
\varepsilon(t)=O\left(\left|\ln(h)\right|^{\beta+3\gamma-4}\right).\]
\end{proof}
\begin{defn}
Let us define the second order approximation of the partial autocorrelation
$\mathbf{a}(t)$ by :\[
\mathbf{a}_{\mathbf{2}}\,:\, t\mapsto\sum_{n\in\mathbb{N}}\left|a_{n}\right|^{2}e^{-itQ_{2}(n)}.\]
And we also define the function \[
\mathbf{\widetilde{a_{2}}}:\, t\mapsto\sum_{n\in\mathbb{N}}\left|a_{n}\right|^{2}e^{-it\left(\mathcal{A}_{h}^{\prime}(2\pi n_{0})2\pi(n-n_{0})+\mathcal{A}_{h}^{\prime\prime}(2\pi n_{0})2\pi^{2}(n-n_{0})^{2}\right)}.\]

\end{defn}
Thus, we get\[
\mathbf{a_{2}}(t)=e^{-it\mathcal{A}_{h}(2\pi n_{0})}\mathbf{\widetilde{a_{2}}}(t)\]
and, we have also\[
\left|\mathbf{a_{2}}(t)\right|=\left|\mathbf{\mathbf{\mathbf{\widetilde{a_{2}}}}}(t)\right|=\left|\sum_{n\in\mathbb{N}}\left|a_{n}\right|^{2}e^{-it\left(\mathcal{A}_{h}^{\prime}(2\pi n_{0})2\pi(n-n_{0})+\mathcal{A}_{h}^{\prime\prime}(2\pi n_{0})2\pi^{2}(n-n_{0})^{2}\right)}\right|.\]

\subsection{Full revival theorem}

Let us start by some notations.
\begin{defn}
Let us define the revival time $T_{rev}=T_{rev}(h,E)$ by : \[
T_{rev}:=\frac{1}{\pi\mathcal{A}_{h}^{\prime\prime}(2\pi n_{0})}\]
and we also denote the integer $N_{h}=N(h)$ defined by :\[
N_{h}:=E\left[\frac{T_{rev}}{T_{hyp}}\right]\in\mathbb{N}.\]

\end{defn}
Hence, there exists an unique real number $\Theta_{h}\in[0,1[$ such
that :\[
\frac{T_{rev}}{T_{hyp}}=N_{h}+\Theta_{h}.\]

\begin{prop}
If we suppose that \[
\lim_{h\rightarrow0}\left(\mathcal{Y}_{h}^{\prime\prime}\circ\mathcal{A}_{h}\right)(2\pi n_{0})=K\]
where $K\neq0$ is a constant, then\[
T_{rev}=\frac{\left|\ln(h)\right|^{3}}{K\left(-V^{''}(0)\right)^{\frac{3}{2}}}+O(1).\]
\end{prop}
\begin{proof}
Since\[
T_{rev}=-\frac{\left(\mathcal{Y}_{h}^{\prime}\circ\mathcal{A}_{h}\right)^{3}(2\pi n_{0})}{\pi\left(\mathcal{Y}_{h}^{\prime\prime}\circ\mathcal{A}_{h}\right)(2\pi n_{0})};\]
if we suppose that \[
\lim_{h\rightarrow0}\left(\mathcal{Y}_{h}^{\prime\prime}\circ\mathcal{A}_{h}\right)(2\pi n_{0})=K;\]
then we deduce \[
T_{rev}=\frac{-\ln(h)^{3}}{K\left(-V^{''}(0)\right)^{\frac{3}{2}}}+O(1).\]

\end{proof}
The full revival theorem is the following :
\begin{thm}
With the previous notations we have :

\textbf{(i)} if $\Theta_{h}=0$; then for all $t\geq0$ : \[
\mathbf{\widetilde{a_{2}}}(t+N_{h}T_{hyp})=\mathbf{\widetilde{a_{2}}}(t+T_{rev})=\mathbf{\widetilde{a_{2}}}(t);\]

\textbf{(ii)} if $\Theta_{h}\in]0,1[$, then for all $t\geq0$\[
\mathbf{\mathbf{\widetilde{a_{2}}}}(t+N_{h}T_{hyp})=\mathbf{\widetilde{a_{2}}}(t)+O\left|\frac{1}{\ln(h)^{2-2\gamma}}\right|;\]
in particular, we have\[
\left|\mathbf{a}_{\mathbf{2}}(t+N_{h}T_{hyp})\right|=\left|\mathbf{a_{2}}(t)\right|+O\left|\frac{1}{\ln(h)^{2-2\gamma}}\right|.\]
\end{thm}
\begin{proof}
For all \textit{$t\geq0$}\[
\mathbf{\widetilde{a_{2}}}(t+N_{h}T_{hyp})={\displaystyle \sum_{n\in\mathbb{N}}\left|a_{n}\right|^{2}e^{-2i\pi\frac{t+N_{h}T_{hyp}}{T_{hyp}}(n-n_{0})}}e^{-2i\pi\frac{(t+T_{rev}-T_{hyp}\Theta_{h})}{T_{rev}}(n-n_{0})^{2}}\]
\[
={\displaystyle \sum_{n\in\mathbb{N}}\left|a_{n}\right|^{2}e^{-2i\pi\frac{t}{T_{hyp}}(n-n_{0})}}e^{-2i\pi\frac{(t-T_{hyp}\Theta_{h})}{T_{rev}}(n-n_{0})^{2}}.\]
If we suppose that \textit{$\Theta_{h}=0$}; then for all $t\geq0$\[
\mathbf{\widetilde{a_{2}}}(t+N_{h}T_{hyp})=\mathbf{\widetilde{a_{2}}}(t).\]
Suppose that $\Theta_{h}\in]0,1[$, then for all \textit{$t\in\mathbb{R}_{+}$}\[
\left|\mathbf{\widetilde{a_{2}}}(t+N_{h}T_{hyp})-\mathbf{\widetilde{a_{2}}}(t)\right|\]
\[
=\left|{\displaystyle \sum_{n\in\mathbb{N}}\left|a_{n}\right|^{2}e^{-2i\pi\frac{t}{T_{hyp}}(n-n_{0})}}e^{-2i\pi\frac{(t-T_{hyp}\Theta_{h})}{T_{rev}}(n-n_{0})^{2}}\right.\]
\[
\left.-{\displaystyle \sum_{n\in\mathbb{N}}\left|a_{n}\right|^{2}e^{-2i\pi\frac{t}{T_{hyp}}(n-n_{0})}}e^{-2i\pi\frac{t}{T_{rev}}(n-n_{0})^{2}}\right|\]
\[
=\left|{\displaystyle \sum_{n\in\mathbb{N}}\left|a_{n}\right|^{2}e^{-2i\pi\frac{t}{T_{hyp}}(n-n_{0})}}e^{-2i\pi\frac{t}{T_{rev}}(n-n_{0})^{2}}\left(e^{2i\pi\Theta_{h}\frac{T_{hyp}}{T_{rev}}(n-n_{0})^{2}}-1\right)\right|\]
\[
\leq\left|{\displaystyle \sum_{n\in\Delta}\left|a_{n}\right|^{2}e^{-2i\pi\frac{t}{T_{hyp}}(n-n_{0})}}e^{-2i\pi\frac{t}{T_{rev}}(n-n_{0})^{2}}\left(e^{2i\pi\Theta_{h}\frac{T_{hyp}}{T_{rev}}(n-n_{0})^{2}}-1\right)\right|\]
\[
+\left|{\displaystyle \sum_{n\in\Gamma}\left|a_{n}\right|^{2}e^{-2i\pi\frac{t}{T_{hyp}}(n-n_{0})}}e^{-2i\pi\frac{t}{T_{rev}}(n-n_{0})^{2}}\left(e^{2i\pi\Theta_{h}\frac{T_{hyp}}{T_{rev}}(n-n_{0})^{2}}-1\right)\right|.\]
For all integer $n\in\Delta$ we have $\left(n-n_{0}\right)^{2}\leq\left|\ln(h)\right|^{2\gamma}$,
hence, for $h\rightarrow0$, we get, for all integer $n\in\Delta$
\[
e^{2i\pi\Theta_{h}\frac{T_{hyp}}{T_{rev}}(n-n_{0})^{2}}=1+O\left(\left|\ln(h)\right|^{2\gamma-2}\right).\]
Thus, for all $t\geq0$\[
\left|{\displaystyle \sum_{n\in\Delta}\left|a_{n}\right|^{2}e^{-2i\pi\frac{t}{T_{hyp}}(n-n_{0})}}e^{-2i\pi\frac{t}{T_{rev}}(n-n_{0})^{2}}\left(e^{2i\pi\Theta_{h}\frac{T_{hyp}}{T_{rev}}(n-n_{0})^{2}}-1\right)\right|\]
\[
\leq O\left(\left|\ln(h)\right|^{2\gamma-2}\right)\sum_{n\in\Delta}\left|a_{n}\right|^{2}\]
\[
\leq O\left(\left|\ln(h)\right|^{2\gamma-2}\right)\sum_{n\in\mathbb{N}}\left|a_{n}\right|^{2}=O\left(\left|\ln(h)\right|^{2\gamma-2}\right).\]
On the other hand, we have\[
\left|{\displaystyle \sum_{n\in\Gamma}\left|a_{n}\right|^{2}e^{-2i\pi\frac{t}{T_{hyp}}(n-n_{0})}}e^{-2i\pi\frac{t}{T_{rev}}(n-n_{0})^{2}}\left(e^{2i\pi\Theta_{h}\frac{T_{hyp}}{T_{rev}}(n-n_{0})^{2}}-1\right)\right|\]
\[
\leq{\displaystyle \sum_{n\in\Gamma}2\left|a_{n}\right|^{2}}=O\left|\frac{1}{\ln(h)^{\infty}}\right|.\]
\end{proof}
\begin{rem}
For $h\rightarrow0$, we have \[
\frac{N_{h}T_{hyp}}{T_{rev}}=\frac{T_{rev}-\Theta_{h}T_{hyp}}{T_{rev}}=1-\frac{\Theta_{h}T_{hyp}}{T_{rev}}\rightarrow1\]
thus, for $h\rightarrow0$ we have also \[
N_{h}T_{hyp}\sim T_{rev}.\]

\end{rem}
The previous theorem show that the function $t\mapsto\mathbf{\widetilde{a_{2}}}(t)$,
is, modulo $O\left(\left|\ln(h)\right|^{2\gamma-2}\right)$, periodic,
with a period equivalent ( for $h\rightarrow0)$ to $T_{rev}$. This
is the full revival phenomenon.

\subsection{Fractional revivals theorem}

The aim of this section is to study the dynamics for time close to
$\frac{p}{q}T_{rev}$ , where $\frac{p}{q}\in\mathbb{Q}.$

\subsubsection{Preliminaries }
\begin{notation}
For all $(p,q)\in\mathbb{Z}\times\mathbb{N}^{*}$ let us consider
the sequence $\left(\sigma_{h}(p,q)\right)$ defined by : \[
\left(\sigma_{h}(p,q)\right)_{n}:=e^{-2i\pi\frac{p}{q}(n-n_{0})^{2}}\;,\, n\in\mathbb{Z}.\]

\end{notation}
The periodicity of this sequence is caracterised by the following
easy proposition. 
\begin{prop}
For all $(p,q)\in\mathbb{Z}\times\mathbb{N}^{*}$, the sequence $\left(\sigma_{h}(p,q)\right)_{n}$
is $\ell$-periodic if and only if the integer $\ell$ satisfy the
equation :\[
\forall m\in\mathbb{Z},\;\frac{2p\ell}{q}m+\frac{p\ell^{2}}{q}\equiv0\;(mod\,1).\]

\end{prop}
Now, we solve this equation.
\begin{prop}
Suppose $p\wedge q=1$, then the set $\mathcal{E}$ of solutions $\ell$
such that :\[
\forall m\in\mathbb{Z},\;\frac{2p\ell}{q}m+\frac{p\ell^{2}}{q}\equiv0\;(mod\,1)\]
is caracterised by :

\textbf{(i)} if $q$ is odd then $\mathcal{E}=\left\{ q\mathbb{Z}\right\} ;$

\textbf{(ii)} si $q$ even and $\frac{q}{2}$ odd then $\mathcal{E}=\left\{ q\mathbb{Z}\right\} ;$

\textbf{(iii)} si $q$ even et $\frac{q}{2}$ even then $\mathcal{E}=\left\{ \frac{q}{2}\mathbb{Z}\right\} .$\end{prop}
\begin{proof}
Let us start by a remark : to be a multiple of $q$ is always a necessary
condition to be a solution. Now, let us study sufficient conditions.
We have\[
\ell\in\mathcal{E}\Leftrightarrow\forall m\in\mathbb{Z},\;\frac{2p\ell}{q}m+\frac{p\ell^{2}}{q}\equiv0\;(mod\,1)\]
\[
\Leftrightarrow\forall m\in\mathbb{Z},\; q|(2p\ell m+p\ell^{2});\]
and, since $p\wedge q=1$, by Gauss lemma we have \[
\ell\in\mathcal{E}\Leftrightarrow\forall m\in\mathbb{Z},\; q|(2\ell m+\ell^{2}).\]
Let $\ell$ an integer solution of $\mathcal{E}$; then if we take
$m=0$ we see that $q|\ell^{2}$, hence there exist $\alpha\in\mathbb{Z}$
such that $\ell^{2}=\alpha q$. On the other hand, since for all $m\in\mathbb{Z}$
we have $q|(2\ell m+\ell^{2})$ we deduce that for all $m\in\mathbb{Z}$,
there exists $\beta_{m}\in\mathbb{Z}$ such that $2\ell m+\ell^{2}=\beta_{m}q$.
Thus, since $\ell^{2}=\alpha q$, we get for all $m\in\mathbb{Z}$
the equality $2\ell m+\alpha q=\beta_{m}q$. In particular, with $m=1$
we deduce that :\[
2\ell=(\beta_{1}-\alpha)q.\]

$\bullet$ If q is odd : then a $\beta_{1}-\alpha$ is even; thus
$\ell=\underbrace{\frac{(\beta_{1}-\alpha)}{2}}_{\in\mathbb{Z}}q$,
and $q|\ell$, thus \textit{$\mathcal{E}=\left\{ q\mathbb{Z}\right\} $.}

$\bullet$If q is even : then $\ell=(\beta_{1}-\alpha)\frac{q}{2}$;
thus $\frac{q}{2}|\ell$. If we write $\ell=k\frac{q}{2}$ with $k\in\mathbb{Z}$,
then \[
\forall m\in\mathbb{Z},\;\frac{2p\ell}{q}m+\frac{p\ell^{2}}{q}=kpm+\frac{pk^{2}q}{4}.\]
Thus \[
\forall m\in\mathbb{Z},\; kpm+\frac{pk^{2}q}{4}\in\mathbb{Z}\Leftrightarrow pk^{2}q\in4\mathbb{Z}.\]
And, since $p$ is odd, finaly we get, with $\ell=k\frac{q}{2}$ :\[
\forall m\in\mathbb{Z},\;\frac{2p\ell}{q}m+\frac{p\ell^{2}}{q}\in\mathbb{Z}\Leftrightarrow k^{2}q\in4\mathbb{Z}.\]
Then, we have two cases :

$\bullet$ if $\frac{q}{2}$ is odd : then $k^{2}q\in4\mathbb{Z}\Leftrightarrow k^{2}$
is even, equivalent to $k$ even, hence $\ell=k\frac{q}{2}$ is a
multiple of $q$, hence \textit{$\mathcal{E}=\left\{ q\mathbb{Z}\right\} $.}

$\bullet$ If $\frac{q}{2}$ is even : then $q=4q'$ where $q'\in\mathbb{Z}^{*}$,
hence for all $k\in\mathbb{Z};\, k^{2}q=k^{2}4q'\in4\mathbb{Z}$ ,
thus \textit{$\mathcal{E}=\left\{ \frac{q}{2}\mathbb{Z}\right\} $.}
\end{proof}
For a period $\ell\in\mathbb{Z}$, let us defined the set of sequences
$\ell-$periodic with a scalar product.
\begin{defn}
For a integer $\ell$$\in\mathbb{Z}^{*}$; let us denote by $\mathfrak{S}_{\ell}(\mathbb{Z})$
set of sequences $\ell-$periodic : \[
\mathfrak{S}_{\ell}(\mathbb{Z}):=\left\{ u_{n}\in\mathbb{C}^{\mathbb{Z}};\,\forall n\in\mathbb{Z},\, u_{n+\ell}=u_{n}\right\} .\]

\end{defn}
So we have the :
\begin{prop}
The application\[
\left\langle \,,\,\right\rangle _{\mathfrak{S}_{\ell}}:\left\{ \begin{array}{cc}
\mathfrak{S}_{\ell}(\mathbb{Z})^{2}\rightarrow\mathbb{C}\\
\\(u,v)\mapsto\left\langle u,v\right\rangle _{\mathfrak{S}_{\ell}}:=\frac{1}{|\ell|}{\displaystyle \sum_{k=0}^{|\ell|-1}u_{k}\overline{v_{k}}}\end{array}\right.\]
is a Hermitean product on the space $\mathfrak{S}_{\ell}(\mathbb{Z})$
.\end{prop}
\begin{proof}
The proof is elementary.\end{proof}
\begin{prop}
Let us consider $\phi_{n}^{k}:=e^{-\frac{2i\pi kn}{\ell}}$ where
$(k,n)\in\mathbb{Z}^{2}$; then the family $\left\{ \left(\phi_{n}^{k}\right)_{n\in\mathbb{Z}}\right\} _{k=0...\ell-1}$
is an orthonormal basis of the space vector $\mathfrak{S}_{\ell}(\mathbb{Z})$.\end{prop}
\begin{proof}
For $\ell>0$, the family \textit{$\left\{ \left(\phi_{n}^{k}\right)_{n\in\mathbb{Z}}\right\} _{k=0...\ell-1}$}
is clearly a familly of $\ell$ vectors of \textit{$\mathfrak{S}_{\ell}(\mathbb{Z})$.}
Next, for all pair $(p,q)\in\left\{ 0...\ell-1\right\} ^{2}$ we have\[
\left\langle \phi^{p},\phi^{q}\right\rangle _{\mathfrak{S}_{\ell}}=\frac{1}{\ell}{\displaystyle \sum_{k=0}^{\ell-1}\phi_{k}^{p}\overline{\phi_{k}^{q}}=\frac{1}{\ell}{\displaystyle \sum_{k=0}^{\ell-1}\left(e^{\frac{2i\pi(q-p)}{\ell}}\right)^{k}}=\delta_{p,q}.}\]
So $\left\{ \left(\phi_{n}^{k}\right)_{n\in\mathbb{Z}}\right\} _{k=0...\ell-1}$
is a orthonormal basis of the space vector \textit{$\mathfrak{S}_{\ell}(\mathbb{Z})$.}
\end{proof}

\subsubsection{The main theorem}
\begin{thm}
For all $(p,q)\in\mathbb{Z}\times\mathbb{N}^{*}$ such that $p\wedge q=1$;
there exists a family of $\ell$ complex numbers $\left(\widetilde{b_{k}}(\ell)\right)_{k\in\left\{ 0...\ell-1\right\} }$
where the integer $\ell\in\mathbb{Z}$ is solution of :\[
\forall m\in\mathbb{Z},\;\frac{2p\ell}{q}m+\frac{pl^{2}}{q}\equiv0\;[1];\]
such that for all $t\in\left[0,\left|\ln(h)\right|^{\alpha}\right]$
we get :\[
\mathbf{\widetilde{a_{2}}}\left(t+\frac{p}{q}N_{h}T_{hyp}\right)={\displaystyle \sum_{k=0}^{\ell-1}\widetilde{b_{k}}(\ell)\mathbf{\widetilde{a_{1}}}\left(t+T_{hyp}\left(\frac{k}{\ell}+\frac{p}{q}N_{h}\right)\right)+O\left(\left|\ln(h)\right|^{\alpha+2\gamma-3}\right)}\]
The numbers $\widetilde{b_{k}}(\ell)$ are called fractionnals coefficients;
and for all $k\in\left\{ 0...\ell-1\right\} $\[
\widetilde{b_{k}}(\ell)=e^{-\frac{2i\pi kn_{0}}{\ell}}b_{k}(\ell)\]
where\[
b_{k}(\ell)=b_{k}(h,\ell)=\left\langle \sigma_{h}(p,q),\phi^{k}\right\rangle _{\mathfrak{S}_{\ell}}=\frac{1}{\ell}{\displaystyle \sum_{n=0}^{\ell-1}e^{-2i\pi\frac{p}{q}(n-n_{0})^{2}}e^{-\frac{2i\pi kn}{\ell}}}.\]
Moreover, if we suppose $\frac{T_{rev}}{T_{hyp}}\in\mathbb{Q}$, then
we have \[
\mathbf{\widetilde{a_{2}}}\left(t+\frac{p}{q}T_{rev}\right)={\displaystyle \sum_{k=0}^{\ell-1}\widetilde{b_{k}}(l)\mathbf{\widetilde{a_{1}}}\left(t+T_{hyp}\left(\frac{k}{\ell}+\frac{p}{q}N_{h}\right)\right).}\]
\end{thm}
\begin{proof}
Let us denote the integer $\widetilde{n}:=n-n_{0}$; and let us consider
$\ell\in\mathbb{Z}$ a solution of the equation $\mathcal{E}$. So,
we have for all $t\geq0$\[
\mathbf{\widetilde{a_{2}}}\left(t+\frac{p}{q}N_{h}T_{cl}\right)={\displaystyle \sum_{n\in\mathbb{N}}\left|a_{n}\right|^{2}e^{-2i\pi\frac{t}{T_{hyp}}\widetilde{n}}}e^{-2i\pi\frac{t}{T_{rev}}\widetilde{n}^{2}}e^{-2i\pi\frac{p}{q}\widetilde{n}N_{h}}e^{-2i\pi\frac{p}{q}\frac{N_{h}T_{hyp}}{T_{rev}}\widetilde{n}^{2}}.\]
And, since $T_{rev}=N_{h}T_{hyp}+\Theta_{h}T_{hyp}$ we have \[
e^{-2i\pi\frac{p}{q}\frac{N_{h}T_{hyp}}{T_{rev}}\widetilde{n}^{2}}=e^{-2i\pi\frac{p}{q}\widetilde{n}^{2}}e^{2i\pi\frac{p}{q}\frac{\Theta_{h}T_{hyp}}{T_{rev}}\widetilde{n}^{2}}\]
and \[
e^{-2i\pi\frac{p}{q}\widetilde{n}^{2}}=e^{-2i\pi\frac{p}{q}(n-n_{0})^{2}}\in\mathfrak{S}_{\ell}(\mathbb{Z});\]
hence, there exist a unique decomposition of the sequence $\left(e^{-2i\pi\frac{p}{q}(n-n_{0})^{2}}\right)_{n}$
on the basis $\left(\phi^{k}\right)_{k\in\left\{ 0...\ell-1\right\} }$:
\[
e^{-2i\pi\frac{p}{q}\widetilde{n}^{2}}=e^{-2i\pi\frac{p}{q}(n-n_{0})^{2}}={\displaystyle \sum_{k=0}^{\ell-1}\left\langle \sigma_{h}(p,q),\phi^{k}\right\rangle _{\mathfrak{S}_{\ell}}\phi_{n}^{k}}\]
\[
={\displaystyle \sum_{k=0}^{\ell-1}b_{k}(\ell)e^{-2i\pi\frac{k}{\ell}n}.}\]
Thus, \foreignlanguage{english}{for all $t\geq0$}\[
\mathbf{\widetilde{a_{2}}}\left(t+\frac{p}{q}N_{h}T_{hyp}\right)\]
\foreignlanguage{english}{\[
={\displaystyle \sum_{n\in\mathbb{N}}\left|a_{n}\right|^{2}e^{-2i\pi\frac{t}{T_{hyp}}\widetilde{n}}}e^{-2i\pi\frac{t}{T_{rev}}\widetilde{n}^{2}}e^{-2i\pi\frac{p}{q}\widetilde{n}N_{h}}\left({\displaystyle \sum_{k=0}^{\ell-1}b_{k}(\ell)e^{-2i\pi\frac{k}{\ell}n}}\right)e^{2i\pi\frac{p}{q}\frac{\Theta_{h}T_{hyp}}{T_{rev}}\widetilde{n}^{2}}\]
\[
={\displaystyle \sum_{n\in\mathbb{N}}{\displaystyle \sum_{k=0}^{\ell-1}b_{k}(\ell)}\left|a_{n}\right|^{2}e^{-2i\pi\frac{t}{T_{hyp}}\widetilde{n}}}e^{-2i\pi\frac{t}{T_{rev}}\widetilde{n}^{2}}e^{-2i\pi\frac{p}{q}\widetilde{n}N_{h}}e^{-2i\pi\frac{k}{\ell}n}e^{2i\pi\frac{p}{q}\frac{\Theta_{h}T_{hyp}}{T_{rev}}\widetilde{n}^{2}}.\]
For all $t\geq0$ we have \[
\mathbf{\mathbf{\widetilde{a_{1}}}}\left(t+T_{hyp}\left(\frac{k}{\ell}+\frac{p}{q}N_{h}\right)\right)={\displaystyle \sum_{n\in\mathbb{N}}\left|a_{n}\right|^{2}e^{-2i\pi\frac{t}{T_{hyp}}\widetilde{n}}}e^{-2i\pi\frac{k}{\ell}\widetilde{n}}e^{-2i\pi\frac{p}{q}N_{h}\widetilde{n}}\]
thus, }\[
{\displaystyle \sum_{k=0}^{\ell-1}\widetilde{b_{k}}(\ell)\mathbf{\mathbf{\widetilde{a_{1}}}}\left(t+T_{hyp}\left(\frac{k}{\ell}+\frac{p}{q}N_{h}\right)\right)}\]
\[
=\sum_{k=0}^{\ell-1}e^{-\frac{2i\pi kn_{0}}{\ell}}b_{k}(\ell){\displaystyle \sum_{n\in\mathbb{N}}\left|a_{n}\right|^{2}e^{-2i\pi\frac{t}{T_{hyp}}\widetilde{n}}}e^{-2i\pi\frac{k}{\ell}\widetilde{n}}e^{-2i\pi\frac{p}{q}N_{h}\widetilde{n}}\]
\[
={\displaystyle \sum_{n\in\mathbb{N}}\sum_{k=0}^{\ell-1}\left|a_{n}\right|^{2}b_{k}(\ell)e^{-2i\pi\frac{t}{T_{hyp}}\widetilde{n}}}e^{-2i\pi\frac{kn}{\ell}}e^{-2i\pi\frac{p}{q}N_{h}\widetilde{n}}.\]
Hence\[
\left|\mathbf{\widetilde{a_{2}}}\left(t+\frac{p}{q}N_{h}T_{hyp}\right)-{\displaystyle \sum_{k=0}^{\ell-1}\widetilde{b_{k}}(\ell)\mathbf{\widetilde{a_{1}}}\left(t+T_{hyp}\left(\frac{k}{\ell}+\frac{p}{q}N_{h}\right)\right)}\right|\]
\[
=\left|{\displaystyle \sum_{n\in\mathbb{N}}\sum_{k=0}^{\ell-1}\left|a_{n}\right|^{2}b_{k}(\ell)e^{-2i\pi\frac{t}{T_{hyp}}\widetilde{n}}}e^{-2i\pi\frac{p}{q}\widetilde{n}N_{h}}e^{-\frac{2i\pi kn}{\ell}}\left(e^{-2i\pi\frac{t}{T_{rev}}\widetilde{n}^{2}}e^{+2i\pi\frac{p}{q}\frac{\Theta_{h}T_{hyp}}{T_{rev}}\widetilde{n}^{2}}-1\right)\right|;\]
with the sets $\Delta,$ $\Gamma$ and by triangular inequality we
get : \[
\leq\left|\sum_{n\in\Gamma}\sum_{k=0}^{\ell-1}\left|a_{n}\right|^{2}b_{k}(\ell)e^{-2i\pi\frac{t}{T_{hyp}}\widetilde{n}}e^{-2i\pi\frac{p}{q}\widetilde{n}N_{h}}e^{-\frac{2i\pi kn}{\ell}}\left(e^{-2i\pi\frac{t}{T_{rev}}\widetilde{n}^{2}}e^{+2i\pi\frac{p}{q}\frac{\Theta_{h}T_{hyp}}{T_{rev}}\widetilde{n}^{2}}-1\right)\right|\]
\[
+\left|\sum_{n\in\Delta}\sum_{k=0}^{\ell-1}\left|a_{n}\right|^{2}b_{k}(\ell)e^{-2i\pi\frac{t}{T_{hyp}}\widetilde{n}}e^{-2i\pi\frac{p}{q}\widetilde{n}N_{h}}e^{-\frac{2i\pi kn}{\ell}}\left(e^{-2i\pi\frac{t}{T_{rev}}\widetilde{n}^{2}}e^{+2i\pi\frac{p}{q}\frac{\Theta_{h}T_{hyp}}{T_{rev}}\widetilde{n}^{2}}-1\right)\right|.\]
Increase the first serie :\[
\left|\sum_{n\in\Gamma}\sum_{k=0}^{\ell-1}\left|a_{n}\right|^{2}b_{k}(\ell)e^{-2i\pi\frac{t}{T_{hyp}}\widetilde{n}}e^{-2i\pi\frac{p}{q}\widetilde{n}N_{h}}e^{-\frac{2i\pi kn}{\ell}}\left(e^{-2i\pi\frac{t}{T_{rev}}\widetilde{n}^{2}}e^{+2i\pi\frac{p}{q}\frac{\Theta_{h}T_{hyp}}{T_{rev}}\widetilde{n}^{2}}-1\right)\right|\]
\[
\leq2\left(\sum_{n\in\Gamma}\left|a_{n}\right|^{2}\right)\left(\sum_{k=0}^{\ell-1}{\displaystyle \left|b_{k}(\ell)\right|}\right)=O\left|\frac{1}{\ln(h)^{\infty}}\right|.\]
Next, for the second serie, observe the following term :\[
e^{-2i\pi\frac{t}{T_{rev}}\widetilde{n}^{2}}e^{+2i\pi\frac{p}{q}\frac{\Theta_{h}T_{hyp}}{T_{rev}}\widetilde{n}^{2}}=e^{-2i\pi\left(t+\Theta_{h}T_{hyp}\frac{p}{q}\right)\frac{1}{T_{rev}}\widetilde{n}^{2}}.\]
For all $n\in\Delta$ and for all \foreignlanguage{english}{$t\in\left[0,\left|\ln(h)\right|^{\alpha}\right]$}
there exist a constant $C>0$ such that \[
\left|\left(t+\frac{\Theta_{h}T_{hyp\,}p}{q}\right)\frac{(n-n_{0})^{2}}{T_{rev}}\right|\leq C\left(\left|\ln(h)\right|^{2\gamma-3+\alpha}+\left|\ln(h)\right|^{2\gamma-2}\right)\leq2C\left|\ln(h)\right|^{2\gamma-2}.\]
Hence, there exist $A>0$ such that for all \foreignlanguage{english}{$t\in\left[0,\left|\ln(h)\right|^{\alpha}\right]$}\[
\left|{\displaystyle \sum_{n\in\Delta}\sum_{k=0}^{\ell-1}\left|a_{n}\right|^{2}b_{k}(\ell)e^{-2i\pi\frac{t}{T_{hyp}}\widetilde{n}}e^{-2i\pi\frac{p}{q}\widetilde{n}N_{h}}e^{-\frac{2i\pi kn}{\ell}}\left(e^{-2i\pi\frac{t\widetilde{n}{}^{2}}{T_{rev}}}e^{+2i\pi\frac{p}{q}\frac{\Theta_{h}T_{hyp}\widetilde{n}{}^{2}}{T_{rev}}}-1\right)}\right|\]
\[
\leq A\left|\ln(h)\right|^{2\gamma-2}\sum_{n\in\Delta}\left|a_{n}\right|^{2}\leq A\left|\ln(h)\right|^{2\gamma-2}\sum_{n=0}^{+\infty}\left|a_{n}\right|^{2}.\]
For finish, the case where $t=0$ and $\frac{T_{rev}}{T_{hyp}}\in\mathbb{Q}$
is clear.\end{proof}
\begin{cor}
We have the equality :\[
\sum_{k=0}^{\ell-1}\left|b_{k}(\ell)\right|^{2}=1.\]
\end{cor}
\begin{proof}
Since $b_{k}(\ell)=\left\langle \sigma_{h}(p,q),\phi^{k}\right\rangle _{\mathfrak{S}_{\ell}}$
and $\left(\phi^{k}\right)_{k\in\left\{ 0...\ell-1\right\} }$ is
a orthomoral basis of the space vector $\mathfrak{S}_{\ell}(\mathbb{Z})$,
by Pythagorean equality we get\[
\sum_{k=0}^{\ell-1}\left|b_{k}(\ell)\right|^{2}=\left\Vert \sigma_{h}(p,q)\right\Vert _{\mathfrak{S}_{\ell}}^{2}=1.\]

\end{proof}
Now, let us examine the case $\frac{p}{q}=1$ and the case $\frac{p}{q}=\frac{1}{2}$. 
\begin{cor}
With the same notation as in the theorem 6.15, we have 

\textbf{(i)} for all $t\in\left[0,\left|\ln(h)\right|^{\alpha}\right]$\[
\mathbf{\widetilde{a_{2}}}(t+N_{h}T_{hyp})=\mathbf{\widetilde{a_{1}}}(t)+O\left(\left|\ln(h)\right|^{\alpha+2\gamma-3}\right);\]
\textbf{(ii)} for all $t\in\left[0,\left|\ln(h)\right|^{\alpha}\right]$\[
\mathbf{\widetilde{a_{2}}}\left(t+\frac{N_{h}T_{hyp}}{2}\right)=\mathbf{\mathbf{\widetilde{a_{1}}}}\left(t+\left(\frac{N_{h}+1}{2}\right)T_{hyp}\right)+O\left(\left|\ln(h)\right|^{\alpha+2\gamma-3}\right);\]
in particular, if $N_{h}$ is odd, then :\[
\mathbf{\mathbf{\widetilde{a_{2}}}}\left(t+\frac{N_{h}T_{hyp}}{2}\right)=\mathbf{\mathbf{\mathbf{\widetilde{a_{1}}}}}(t)+O\left(\left|\ln(h)\right|^{\alpha+2\gamma-3}\right)\]
and if $N_{h}$ is even, then :\[
\mathbf{\mathbf{\widetilde{a_{2}}}}\left(t+\frac{N_{h}T_{hyp}}{2}\right)=\mathbf{\widetilde{a_{1}}}\left(t+\frac{T_{hyp}}{2}\right)+O\left(\left|\ln(h)\right|^{\alpha+2\gamma-3}\right).\]
\end{cor}
\begin{proof}
In the case (i) : $p=1,\, q=\ell=1;$ thus by the previous theorem
we get, for all \textit{$t\in\left[0,\left|\ln(h)\right|{}^{\alpha}\right]$}
:\[
\mathbf{\mathbf{\widetilde{a_{2}}}}\left(t+\frac{p}{q}N_{h}T_{hyp}\right)={\displaystyle \widetilde{b_{0}}(1)\mathbf{\widetilde{a_{1}}}\left(t+T_{hyp}N_{h}\right)+O\left(\left|\ln(h)\right|^{\alpha+2\gamma-3}\right)}\]
\[
={\displaystyle \widetilde{b_{0}}(1)\mathbf{\mathbf{\widetilde{a_{1}}}}(t)+O\left(\left|\ln(h)\right|^{\alpha+2\gamma-3}\right)}\]
and $b_{0}(1)=\frac{1}{1}e^{-2i\pi n_{0}^{2}}=1$, thus\[
\widetilde{b_{0}}(1)=1.\]
Next, for the case (ii) : $p=1,\, q=\ell=2$ , then \[
\mathbf{\widetilde{a_{2}}}\left(t+\frac{p}{q}N_{h}T_{hyp}\right)=\]
\[
{\displaystyle \widetilde{b_{0}}(2)\mathbf{\mathbf{\widetilde{a_{1}}}}\left(t+\frac{T_{hyp}N_{h}}{2}\right)+\widetilde{b_{1}}(2)\mathbf{\mathbf{\widetilde{a_{1}}}}\left(t+T_{hyp}\left(\frac{1}{2}+\frac{N_{h}}{2}\right)\right)+O\left(\left|\ln(h)\right|^{\alpha+2\gamma-3}\right).}\]
We have\[
b_{0}(2)=\frac{1}{2}e^{-2i\pi\frac{1}{2}n_{0}^{2}}+\frac{1}{2}e^{-2i\pi\frac{1}{2}(1-n_{0})^{2}}=0\]
thus \[
\widetilde{b_{0}}(2)=0.\]
 And\[
b_{1}(2)=\frac{1}{2}e^{-2i\pi\frac{1}{2}n_{0}^{2}}e^{\frac{2i\pi0}{2}}+\frac{1}{2}e^{-2i\pi\frac{1}{2}(1-n_{0})^{2}}e^{\frac{2i\pi}{2}}\]
\[
=\frac{1}{2}\left(e^{-i\pi n_{0}{}^{2}}+(-1)e^{-i\pi(n_{0}-1)^{2}}\right)=(-1)^{n_{0}}\]
then \[
\widetilde{b_{1}}(2)=e^{-i\pi n_{0}}(-1)^{n_{0}}=1.\]

\end{proof}

\subsection{Explicit values of modulus for revivals coefficients}

About the coefficients $\widetilde{b_{k}}(\ell)$ we just know that
:\[
{\displaystyle \sum_{k=0}^{\ell-1}\left|\widetilde{b_{k}}(\ell)\right|^{2}}=1.\]
But we can say more. We can calculate $\left|\widetilde{b_{k}}(\ell)\right|=\left|b_{k}(\ell)\right|$
for all $k$. From the proposition 6.11 we consider two cases : the
case $\ell=q$ and the case $\ell=\frac{q}{2}$.

\subsubsection{Case $\ell=q$}

If the integer $q$ is odd we have :
\begin{thm}
\textbf{{[}Lab2{]}.} For all integers $p$ and $q$, such that $p\wedge q=1$
and $q$ odd, then for all $k\in\left\{ 0...q-1\right\} $ we get
:\[
\left|b_{k}(q)\right|^{2}=\frac{1}{q}.\]

\end{thm}
And in the case $q$ even :
\begin{thm}
\textbf{{[}Lab2{]}}. For all integers $p$ and $q$, such that $p\wedge q=1$
and $q$ even, then for all $k\in\left\{ 0...q-1\right\} $ we get
:\[
If\:\frac{q}{2}\, is\, even,\, then:\;\left|b_{k}(q)\right|^{2}=\left\{ \begin{array}{cc}
\frac{2}{q}\;\; if\; k\; is\; even\\
\\0\;\; else.\end{array}\right.\]
\[
If\:\frac{q}{2}\, is\, odd,\, then:\;\left|b_{k}(q)\right|^{2}=\left\{ \begin{array}{cc}
0\;\; if\; k\; is\; even\\
\\\frac{2}{q}\; else.\end{array}\right.\]

\end{thm}

\subsubsection{Case $\ell=\frac{q}{2}$}

It is the case $q\in4\mathbb{Z}$. 
\begin{thm}
\textbf{{[}Lab2{]}}. For all integers $p$ and $q$, such that $p\wedge q=1$
and $q\in4\mathbb{Z}^{*}$; for all $k\in\left\{ 0...\frac{q}{2}-1\right\} $
we have \[
\left|b_{k}\left(\frac{q}{2}\right)\right|^{2}=\frac{2}{q}.\]

\end{thm}

\subsection{Comparison between time scale approximation, hyperbolic period and
revivals periods}

Since $\gamma<1$ we have $1<3-2\gamma,$ thus there exist $\alpha$
such that $\alpha\in\left]1,3-2\gamma\right[$; hence : \[
\left|\ln(h)\right|<\left|\ln(h)\right|^{\alpha}<\left|\ln(h)\right|^{3-2\gamma}.\]
So we can choose $\alpha$ such that : for $h$ small enough :\[
\left[0,T_{hyp}\right]\subset\left[0,\left|\ln(h)\right|^{\alpha}\right].\]
Next, with $\gamma<\frac{1}{3}$ we get $3<4-3\gamma$ and there exist
$\beta$ such that $\beta\in\left]1,4-3\gamma\right[$ hence : \[
\left|\ln(h)\right|^{3}<\left|\ln(h)\right|^{\beta}<\left|\ln(h)\right|^{4-3\gamma};\]
so, for $h$ small enough we have :\[
\left[0,T_{rev}\right]\subset\left[0,\left|\ln(h)\right|^{\beta}\right].\]

\selectlanguage{english}%

\vspace{1.5cm}\textbf{\large Olivier Labl\'ee}{\large \par}

\hspace{-0.5cm}

\hspace{-0.5cm}Universit\'{e} Grenoble 1-CNRS

\hspace{-0.5cm}Institut Fourier

\hspace{-0.5cm}UFR de Math\'ematiques

\hspace{-0.5cm}UMR 5582

\hspace{-0.5cm}BP 74 38402 Saint Martin d'H\`eres 

\hspace{-0.5cm}mail: \textcolor{blue}{lablee@ujf-grenoble.fr}

\hspace{-0.5cm}http://www-fourier.ujf-grenoble.fr/\textasciitilde{}lablee/
\end{document}